\documentclass[reqno,10pt,a4paper]{amsart}

\usepackage{subfigure}
\usepackage{amsmath,amssymb,amsthm,bbm,marvosym}
\usepackage{epsfig,psfrag,color}
\usepackage{enumitem}
\setenumerate{leftmargin=*}

\let\subset\subseteq 
\let\eps\varepsilon
\let\rho\varrho
\def\dcup{\dot\cup}

\def\neighbor{\mathrm{N}}

\def\parent{\mathrm{Par}}

\def\Szemeredi{Szemer\'edi}
\def\Lovasz{Lov\'asz}

\newcounter{main}
\newcounter{bipfar}

\newtheorem{theorem}{Theorem}
\newtheorem{lemma}[theorem] {Lemma}   
   
\newtheorem{conjecture}[theorem] {Conjecture}

\theoremstyle{remark} 
\newtheorem{remark}[theorem] {Remark} 
\newtheorem{AuxiliaryMain}{Claim}[main]
\newtheorem{AuxiliaryBipfar}{Claim}[bipfar]
\newtheorem{AuxiliaryCl}{Claim}[theorem]

\newcommand{\By}[2]{\overset{\mbox{\tiny{#1}}}{#2}}
\newcommand{\ByRef}[2]{   \By{\eqref{#1}}{#2} }

\newcommand{\gByRef}[1]{  \ByRef{#1}{>} }
\newcommand{\leByRef}[1]{ \ByRef{#1}{\le} }

\def\Aut{\mathrm{Aut}} 
\def\fvc{\tau^*} 

\renewcommand{\leq}{\leqslant}
\renewcommand{\le}{\leqslant}
\renewcommand{\geq}{\geqslant}
\renewcommand{\ge}{\geqslant}

\renewcommand{\epsilon}{\varepsilon}

\let\oldmarginpar\marginpar
\renewcommand\marginpar[1]{\-\oldmarginpar[\raggedleft\footnotesize #1]%
{\raggedright\footnotesize #1}}

\def\COMMENT#1{}
\let\COMMENT=\footnote

\title{Hamilton cycles in dense vertex-transitive graphs}
\author{Demetres Christofides}
\address{School of Sciences,
       UCLan Cyprus,
       12--14 University Avenue, 
       Pyla, 7080 Larnaka, Cyprus}
\email{dchristofides@uclan.ac.uk}       
\author{Jan Hladk\'y}
\address{DIMAP and Department of Computer Science,
       University of Warwick,
       Coventry, CV4~7AL, UK}
\email{honzahladky@gmail.com}       
\author{Andr\'as M\'ath\'e}
\address{Mathematics Institute,
       University of Warwick,
       Coventry, CV4~7AL, UK}
\email{A.Mathe@warwick.ac.uk}       
\thanks{Research of DC and JH is supported by DIMAP,
EPSRC award EP/D063191/1. AM is supported by a Leverhulme Trust Early Career Fellowship. JH is an EPSRC Fellow.}
\date{\today}

\begin{document}

\begin{abstract}
A famous conjecture of Lov\'asz states that every connected vertex-transitive graph contains a Hamilton path. In this article we confirm the conjecture in the case that the graph is dense and sufficiently large. In fact, we show that such graphs contain a Hamilton cycle and moreover we provide a polynomial time algorithm for finding such a cycle. 
\end{abstract}
\maketitle


\section{Introduction}

The decision problems of whether a graph contains a Hamilton cycle or a
Hamilton path are two of the most famous NP-complete problems, and so it is
unlikely that there exist good characterizations of such graphs. For this
reason, it is natural to ask for sufficient conditions which ensure the existence
of a Hamilton cycle or a Hamilton path. To this direction, the following
well-known conjecture of \Lovasz\ is still wide open.

\begin{conjecture}\label{conj:main}
Every connected vertex-transitive graph has a Hamilton path.
\end{conjecture}
Let us recall that a graph is \emph{vertex-transitive} if its automorphism group acts transitively upon its vertices.

In contrast to common belief, \Lovasz\
in~1969~\cite{ConferenceCalgary69} asked for the construction of a
connected vertex-transitive graph containing no Hamilton path. Traditionally however, the \Lovasz\ conjecture is always stated in the positive.

At the moment no counterexample is known. Moreover, there are only five known
examples of connected vertex-transitive graphs having no Hamilton cycle. These
are $K_2$, the Petersen graph, the Coxeter graph and the graphs obtained from
the Petersen and Coxeter graphs by replacing every vertex with a triangle. Apart
from $K_2$, the other four examples are not Cayley graphs and this leads to the
conjecture that every connected Cayley graph on at least three vertices is
Hamiltonian. Similarly as with Conjecture~\ref{conj:main} this is
now folklore, and its origin may be difficult to trace back, but probably the first
conjecture in this direction is due to Thomassen (see e.g. \cite{Ber79}), and asserts
that there are only finitely many connected vertex-transitive graphs that do
not have a Hamilton cycle. At the moment however, the best known general result
which is due to Babai~\cite{Babai79} states that every connected vertex-transitive graph on $n$ vertices has a cycle of length at least $\sqrt{3n}$.

The conjecture has attracted a lot of interest from researchers and there is no common agreement as to its validity. For example, in the negative direction, Babai~\cite{Babai95} conjectured that there is an absolute constant $c > 0$ and infinitely many connected Cayley graphs $G$ without cycles of length greater than $(1-c)|G|$.

We will omit any further overview of the vast research these questions have motivated, referring the reader to the following surveys~\cite{Witte&Gallian84,Curran&Gallian96,KutnarMarusic:HamiltonSurvey,Pak&Radoicic09} and their references.

In this paper we prove that every sufficiently large dense connected vertex-transitive graph is Hamiltonian.

\begin{theorem}\label{thm:main}
For every $\alpha>0$ there exists an $n_0$ such that every connected vertex-transitive graph on $n \geqslant n_0$ vertices of valency at least $\alpha n$ contains a Hamilton cycle.
\end{theorem}

\addtocounter{main}{\value{theorem}}


\subsection{Relation to previous results.}

As said above, we do not aim to survey results related to Conjecture~\ref{conj:main}. However, it turns out that Theorem~\ref{thm:main} is implied in several settings by other results. We want to describe these and pinpoint some situations when the Hamiltonicity given by Theorem~\ref{thm:main} was not known before. We will restrict the discussion to the family of Cayley graphs.

Recall that Fleischner's Theorem~\cite{Fleischner:Square} asserts that the
(distance-)square of a 2-connected graph is Hamiltonian. Suppose that $G$ is a
connected Cayley graph over a group $\Gamma$ with a generating set $X$.
2-connectedness is easily shown to be implied by connectedness for Cayley graphs. If we find a set  $Y\subset X$ which
generates $\Gamma$, and such that $Y^2\subset X$, then Fleischner's
Theorem applies and the Hamiltonicity of $G$ follows. This is a `typical'\footnote{In the sense that most examples that come to mind are of this sort} situation when $X$ is dense in $\Gamma$. However, there are examples, when the set $Y$ does not exist.

There are two important classes of groups where Hamiltonicity of the corresponding Cayley graph follows by other methods. One class is abelian groups. In the abelian setting, the Hamiltonicity of the Cayley graph is known for all generating sets. The argument has been pushed further by Pak and Radoi\v{c}i{\'c}~\cite{Pak&Radoicic09} to groups which are close to abelian. Another important class is groups with no non-trivial irreducible representations of low dimension. This family for example, contains all non-abelian simple groups. For these groups, Gowers~\cite{Gowers:QuasirandomGrops} proved that the corresponding Cayley graph is quasirandom (in the sense of Chung-Graham-Wilson~\cite{ChGraWi89:Quasirandom}), no matter what the set $X$ of generators is taken to be (provided that $X$ is dense). In this case, the Hamiltonicity follows from the well-known fact (see e.g.~\cite[Proposition 4.19]{KriSu:Pseudorandom}) that dense pseudorandom graphs are Hamiltonian. However, there are groups which are very far from abelian and yet have non-trivial low-dimensional representations. Soluble groups are one such example.


\subsection{Overview}

Here is an overview of the rest of the paper. Section~\ref{sec:notation} contains some
notation that we are going to use. Our proof will use \Szemeredi's Regularity
Lemma. In using the Regularity Lemma, we would like some properties of
the original graph $G$ to be inherited by the reduced graph obtained from the
application of the lemma. In Section~\ref{sec:matching} we discuss some results from matching theory in this
direction. These results will enable us to show that the reduced graph (after a
minor modification) contains an almost perfect  matching. In Section~\ref{sec:connectivity}
we discuss two non-standard notions of connectivity: robustness and
iron connectivity. The main result of Section~\ref{sec:connectivity} is
Theorem~\ref{lem:goodIslands} which says that $G$ can be partitioned into a
bounded number of isomorphic vertex-transitive pieces each of which is
iron connected. This is a much stronger notion than the standard notion of
vertex connectivity. In particular, iron connectivity is inherited by the
reduced graph as well. It will turn out that if $G$ `looks very much like a
bipartite graph' then there are some additional difficulties that need to be
overcome. In Section~\ref{sec:bipartite} we quantify what we mean by the phrase
`looks very much like a bipartite graph' and prove that in this case the vertex
set of $G$ can be partitioned into two equal parts such that every automorphism
of $G$ respects this partition. In Section~\ref{sec:regularity} we collect all the tools
needed for the application of the Regularity Lemma. In Section~\ref{sec:HCinIronConnected} we apply
the Regularity Lemma to show that every sufficiently large iron connected vertex-transitive graph
contains a Hamilton cycle. In fact, we will need and prove a somewhat stronger
property. Finally, in Section~\ref{sec:proof} we put all the pieces together. We first
partition $G$ into the bounded number of vertex-transitive, iron connected
pieces, then find a Hamilton cycle in each of these pieces, and then show how to
glue these pieces together. It turns out that what we need for the glueing is
not Hamilton cycles but rather more general objects which we call $\ell$-pathitions. Their
existence is also guaranteed from our work in Section~\ref{sec:HCinIronConnected}.

It turns out that all the steps of our proof of Theorem~\ref{thm:main} can be
performed algorithmically. In Section~\ref{sec:algorithmic} we discuss how to
turn the proof into a polynomial time algorithm for finding a Hamilton cycle in
dense vertex-transitive graphs.


\section{Notation and preliminaries}\label{sec:notation}

Given a positive integer $m$ we will often denote the set $\{1,\ldots,m\}$ of the first $m$ positive integers by $[m]$.

If every vertex of a graph $G$ has the same degree $k$ then we say that $G$ has
\emph{valency $k$}, and write $\deg(G) = k$. For a set $E'\subset E(G)$ we write
$\Delta(E')$ for the maximum degree of the subgraph induced by $E'$. Further,
for two disjoint sets $A,B\subset V(G)$ we write $\Delta_G(A,B)$ for the maximum
degree of the bipartite graph $G[A,B]$. For a vertex $v \in V(G)$ and a subset
$A \subseteq V(G)$ we write $\neighbor_{A}(v)$ for the set of neighbours of $v$
which lie in $A$. We denote the size of $\neighbor_{A}(v)$ by $\deg(v,A)$.

We denote the automorphism group of $G$ by $\Aut(G)$. We will usually denote the elements of $\Aut(G)$ by $f$ or $g$. 

Recall that a graph $G$ is \emph{Hamilton-connected} if for any pair of distinct
vertices $x,y$ there is a Hamilton path with $x$ and $y$ as terminal vertices.
Another important connectivity notion is that of linkedness: $G$ is
\emph{$\ell$-linked} if for any set of distinct vertices
$x_1,\ldots,x_\ell,y_1,\ldots,y_\ell\in V(G)$ there exist vertex-disjoint paths
$P_1,\ldots,P_\ell$ such that $x_i$ and $y_i$ are terminal vertices of $P_i$.
For our proof of Theorem~\ref{thm:main}, we will need a combination of the two
notions above. Given a graph $G$ and a subset $U$ of the vertex set of $G$, we say that $G$ is \emph{$\ell$-pathitionable with
exceptional set $U$} if for any $\ell'\in[\ell]$, and for any set of distinct
vertices $x_1,\ldots,x_{\ell'},y_1,\ldots,y_{\ell'}\in V(G)\setminus U$ there
exist vertex-disjoint paths $P_1,\ldots,P_{\ell'}$ such that $x_i$ and $y_i$ are
terminal vertices of $P_i$. Furthermore, we require that the paths
$P_1,\ldots,P_{\ell'}$ cover all the vertices of $G$. So a graph is
$1$-pathitionable with exceptional set $\emptyset$ if and only if it is
Hamilton-connected.

Observe that for example the complete bipartite graph $K_{n,n}$ is not
1-pathitionable. Indeed, we cannot connect two vertices of the same colour class
of $K_{n,n}$ by a Hamilton path. Yet, we will need to deal with graphs which are
bipartite or even almost bipartite. To this end we introduce a modification of
pathitionability to bipartite setting. Suppose that a graph $G$ together with a
partition $V(G)=A\dcup B$ is given. We say that $G$ is
\emph{$\ell$-bipathitionable with exceptional set $U$ with respect to the partition $A\dcup B$} if for any $\ell'\in[\ell]$, and for any set of distinct vertices $x_1,\ldots,x_{\ell'},y_1,\ldots,y_{\ell'}\in V(G)\setminus U$ such that
\begin{equation}\label{eq:BIrequirement}
 |\{x_1,\ldots,x_{\ell'},y_1,\ldots,y_{\ell'}\}\cap
A|=|\{x_1,\ldots,x_{\ell'},y_1,\ldots,y_{\ell'}\}\cap B|
\end{equation}
there exist vertex-disjoint paths $P_1,\ldots,P_{\ell'}$ such that $x_i$ and $y_i$ are terminal vertices of $P_i$. Furthermore, we require that the paths $P_1,\ldots,P_{\ell'}$ cover all the vertices of $G$.

Suppose that $\mathcal S=\{P_1,\ldots, P_\ell\}$ is a system of vertex-disjoint
paths in a graph $G$. We then say that a system of paths $\mathcal
S'=\{P'_1,\ldots,P'_\ell\}$ is an \emph{extension of $\mathcal S$} if the paths
$P'_i$ are vertex-disjoint, and for each $i\in[\ell]$ we have $V(P'_i)\supset
V(P_i)$, and $P_i$ and $P'_i$ have the same endvertices. If $\mathcal S'$ covers
all the vertices of $G$ then we say that $\mathcal S'$ is a \emph{complete
extension}.

Given a graph $G$ and a natural number $\ell$, the \emph{$\ell$-blow-up of $G$}, denoted $\ell\times G$ is the graph in which every vertex of $G$ is replaced by an independent set of size $\ell$, and each edge of $G$ is replaced by a complete bipartite graph between the two corresponding independent sets.

As an auxiliary tool we will need to work with \emph{digraphs} as well. For basic terminology about digraphs we refer the reader to~\cite{DigraphBook}. In particular we do not allow loops or multiple edges. (We do however allow edges between the same two vertices which have different direction.) Recall that a digraph $G$ is strongly connected if for any pair of distinct vertices $a,b\in V(G)$ there is a directed walk from $a$ to $b$. We will also need the following extension of the notion of strong connectedness: we say that a digraph $D$ is \emph{$\ell$-strongly connected} if for every set $U\subset V(D)$, $|U|\le \ell$ and for any pair of distinct vertices $a,b\in V(G)\setminus U$ there exists a directed walk from $a$ to $b$ avoiding $U$. 

Given a (finite) set $X$ and a function $f:X \to \mathbb{R}$ we will write $\|f\|_1$ for the sum $\sum_{x \in X} |f(x)|$. 

Finally, to avoid unnecessarily complicated calculations, we will sometimes omit floor and ceiling signs and treat large numbers as if they were integers.


\section{Some matching theory}\label{sec:matching}

Let us recall that a function $f:V\rightarrow [0,1]$ is a \emph{fractional vertex cover} of a graph $G=(V,E)$ if $f(x)+f(y)\ge 1$ for every $xy\in E$. We write $\fvc(G)$ for the weight of the minimum fractional vertex cover, i.e.
\[
\fvc(G)=\min\{\|f\|_1 : \mbox{$f$ is a fractional vertex cover of $G$}\}\;.
\]
A function $M:E\rightarrow [0,1]$ is  a \emph{fractional matching} of a graph $G=(V,E)$ if for every $v\in V$ we have $\sum_{e \ni v}M(e)\le 1$, where the summation is taken over all edges $e\in E$ containing the vertex $v$. We write $\nu^*(G)$ for the weight of the maximum fractional matching, i.e.
\[
\nu^*(G) = \max\{\|M\|_1 : \text{$M$ is a fractional matching of $G$}\}.
\]
The fractional matching $M$ is said to be \emph{half-integral} if $M(e)\in\{0,\frac12,1\}$ for every $e\in E$.

It is easy to see that for every graph $G$ we have $\fvc(G) \geqslant \nu^*(G)$. The duality of linear programming guarantees that in fact we have equality. Moreover, the 
half-integrality property of fractional matchings (cf.~\cite[Theorem~30.2]{Schr03a}) says that there is a half-integral matching with weight $\nu^*(G)$. 

\begin{theorem}\label{thm:HalfIntegralMatching} $ $
\begin{enumerate}[label=(\alph{*})]
\item\label{it:Ma} For every graph $G$ we have $\fvc(G) = \nu^*(G)$.
\item\label{it:Mb} For every graph $G$ there is a half-integral matching $M$ of
$G$ with $\|M\|_1 = \nu^*(G)$.
\end{enumerate}
\end{theorem}

The next lemma asserts that removal of a small fraction of edges from a
vertex-transitive graph $G$ does not decrease $\fvc(G)$ much.

\begin{lemma}\label{lem:coverAfter}
Let $G$ be a vertex-transitive graph on $n$ vertices. Suppose $G'$ is a spanning subgraph of $G$ such that $e(G')\ge (1-\delta)e(G)$. Then $\fvc(G') \geqslant (1 - \delta)\fvc(G)$.
\end{lemma}

\begin{proof}
Let $f:V(G)\rightarrow [0,1]$ be an arbitrary fractional vertex cover of $G'$. To prove the lemma, it suffices to show that there is a function
$f':V(G)\rightarrow [0,1]$ such that
\begin{enumerate}[label=(\alph{*})]
\item\label{it:A} $\|f\|_1=\|f'\|_1$;
\item\label{it:B} $f'(x)+f'(y) \geqslant 1-\delta$ for every edge $xy \in
E(G)$.
\end{enumerate}
Indeed, if the above hold then the function $g:V(G) \to [0,1]$ defined by $g(x) = f'(x)/(1-\delta)$ is a fractional vertex cover of $G$ with $(1-\delta)\|g\|_1 = \|f\|_1$ and the claim of the lemma follows.

To show that such an $f'$ exists, we define 
\[
f'(v)=\frac1{|\Aut(G)|}\sum_{g\in\Aut(G)}f(g(v))\;.
\]
As $G$ is vertex-transitive, $f'$ is constant. Further,~\ref{it:A} is satisfied. Suppose for
contradiction that~\ref{it:B} fails for some edge $xy$ of $G$. Since $f'$ is
constant, we get that~\ref{it:B} fails for every edge of $G$. Thus,
\begin{equation}\label{eq:plug}
\sum_{uv\in E(G)}(f'(u)+f'(v))<(1-\delta)e(G)\le e(G')\le \sum_{uv\in
E(G')}(f(u)+f(v))\;,
\end{equation}
where the last inequality follows from the fact that $f$ is a fractional vertex cover of $G'$. Plugging the defining formula for~$f'$ in~\eqref{eq:plug} we get
\[
\sum_{g\in\Aut(G)}\sum_{uv\in E(G)}(f(g(u))+f(g(v)))<
\sum_{g\in\Aut(G)}\sum_{uv\in E(G')}(f(u)+f(v))\;,
\]
Observe that the sum $\sum_{uv\in E(G)}(f(g(u))+f(g(v)))$ does not depend on $g$. Therefore, $\sum_{uv\in E(G)}(f(u)+f(v))<\sum_{uv\in E(G')}(f(u)+f(v))$, a contradiction.
\end{proof}

The following lemma asserts that $\fvc(G)=\frac n2$ for every non-empty vertex-transitive graph of order $n$. This is easy and well-known; nevertheless we include a proof for completeness.

\begin{lemma}\label{lem:CoverVertexTr}
Suppose that $G$ is a vertex-transitive graph of order $n$ and at least one
edge. Then $\fvc(G)=\frac n2$. 
\end{lemma}

\begin{proof}
The constant one-half function is a fractional vertex cover of $G$, thus
establishing $\fvc(G)\le \frac n2$.

Suppose for contradiction that there exists a fractional vertex cover $f:V(G)\rightarrow [0,1]$ such that $\|f\|_1<\frac n2$. The function $f':V(G)\rightarrow [0,1]$ defined by $f'(v)=\frac1{|\Aut(G)|}\sum_{g\in\Aut(G)}f(g(v))$ is a constant function, which is a fractional vertex cover. Since $\|f'\|_1=\|f\|_1<\frac n2$, we have $f'(v)<\frac12$ for each $v\in V(G)$. In particular, $f'(x)+f'(y)<1$ for an edge $xy\in E(G)$, a contradiction.
\end{proof}

The next lemma asserts that 2-blow-up graphs contain an integral matching which is
twice the weight of the maximum fractional matching of the original graph.
\begin{lemma}\label{lem:matchingBlowUp}
There exists a matching of weight $2\nu^*(H)$ in the graph $2\times H$.
\end{lemma}
\begin{proof} 
Suppose that each vertex
$v$ in $H$ was replaced by two vertices $v^1$ and $v^2$ in the graph $2\times
H$.

Consider a half-integral matching $M$ in the graph $H$ of weight $\nu^*(H)$.
Such a matching exists by Theorem~\ref{thm:HalfIntegralMatching}\ref{it:Mb}. We
now construct an integral matching (i.e.~a matching) $M'$ in $2\times H$ of
weight $2\nu^*(H)$ as follows: For any edge $uv$ with weight~1 in $M$,
we add the edges $u^1v^1$ and $u^2v^2$ in $M'$. The set of edges with weight
$\frac12$ in $M$ form a subgraph of $R$ which is a union of paths and cycles.
For every such path $v_1 \cdots v_r$ we add in $M'$ all edges of the form
$v_j^sv_{j+1}^s$ with $1 \leqslant s \leqslant 2, 1 \leqslant j \leqslant
r-1$ and $j+s$ even. Finally, for every such cycle $v_1 \cdots
v_rv_1$ we add in $M'$ all edges of the form $v_j^sv_{j+1}^{s}$
with $1 \leqslant s \leqslant 2, 1 \leqslant j \leqslant r-1$ and $j+s$ even,
together with either the edge $v_r^1 v_1^2$ if $r$ is odd or the edge
$v_r^2 v_1^2$ if $r$ is even. It is immediate by the construction that
$M'$ is indeed a matching of $2\times H$ of weight $\|M'\|_1=2\|M\|_1=
2\nu^*(H)$.
\end{proof}

The next lemma says that the property of containing a large matching is
inherited by the reduced graph as well. Here we formulate it without referring
to the Regularity lemma (and the notion of the reduced graph, both
notions introduced only in Section~\ref{sec:regularity}).
\begin{lemma}\label{lem:LargeMatchingInherited}
Suppose that a graph $\tilde{R}$ is given and let $\tilde{G}$ be a subgraph
of its $m$-blow-up. Then
$\nu^*(\tilde{R})\ge\frac{\nu^*(\tilde{G})}{m}$.
\end{lemma}
\begin{proof}
Suppose that a fractional matching $M$ in $\tilde{G}$ is given. We can then
define a fractional matching $M_{\tilde{R}}$ in $\tilde{R}$ by defining its
weight on an edge $AB\in E(\tilde{R})$ as 
$$M_{\tilde{R}}(AB)=\frac1{m}\sum_{a\in A,b\in B,ab\in
E(\tilde{G})}M(ab)\;.$$ This is indeed a fractional matching as 
for each $A\in V(\tilde R)$ we have
\[
\sum_{B: AB\in E(\tilde R)}M_{\tilde{R}}(AB) = \frac{1}{m}\sum_{a \in
A}\sum_{b: ab \in E(\tilde{G})}M(ab) \leqslant \frac{1}{m}\sum_{a \in
A}\sum_{b:ab \in E(\tilde{G})}M(ab)\le \frac{1}{m}\sum_{a \in
A}1\le 1\;.
\] 
Moreover, 
\[ 
\|M_{\tilde R}\|_1 = \frac 1m \sum_{e \in E(\tilde G)} M(e)
\;,
\]
and the lemma follows.
\end{proof}


\section{Robustness and iron connectivity}\label{sec:connectivity}

We introduce two non-standard notions of connectivity: robustness and iron connectivity. These notions turn out to be suitable in combination with  the Regularity Lemma --- roughly speaking, when a graph has high iron connectivity, then the reduced graph corresponding to it also has high iron connectivity.

We say that a graph $G$ is \emph{$\ell$-robust} if $G$ remains connected even after removal of an arbitrary set $E'\subset E(G)$ with $\Delta(E')\le \ell$. We say that $G$ is \emph{$\ell$-iron} if $G$ stays connected after simultaneous removal of an arbitrary edge-set $E'\subset E(G)$ with $\Delta(E')\le \ell$ and an arbitrary vertex-set $U\subset V(G)$ with $|U|\le \ell$.

Our main aim in this chapter is to show that every dense vertex-transitive graph can be partitioned into not too many isomorphic vertex-transitive subgraphs which have high iron connectivity. This is stated in the following theorem.

\begin{theorem}\label{lem:goodIslands}
For every $\alpha>0$ there exist $\beta,R,N_0>0$ such that the following holds: Suppose $G$ is a vertex-transitive graph of order $n>N_0$ and valency at least $\alpha n$. Then there exists a partition $V(G)=V_1 \dcup \cdots \dcup V_r$ into $r<R$  parts such that all the graphs $G[V_i]$ are isomorphic to a graph $G'$ which is vertex-transitive and $(\beta n)$-iron. Furthermore, for each $g \in \Aut(G)$ and each $1 \leqslant j \leqslant r$ we have $g(V_j) \in \{V_1,\ldots,V_r\}$. 
\end{theorem}
A typical example of a connected vertex-transitive graph $G$ with very low iron connectivity (and even robustness) is a graph formed by two disjoint cliques of order $n/2$, say on vertex sets $V_1$ and $V_2$, with a perfect matching between $V_1$ and $V_2$. The sets $V_1$ and $V_2$ are likely to be the decomposition of $G$ given by Theorem~\ref{lem:goodIslands} and indeed this is the decomposition our proof would give.

The first step towards the proof of the above theorem would be to gather
together vertices of $G$ which cannot be separated from the removal of an edge
set of small maximum degree. To this end, given two vertices $u$ and $v$ of $G$
we say that $u$ and $v$ are \emph{$\ell$-robustly adjacent} if whenever we
remove from $G$ an arbitrary set $E'\subset E(G)$ with $\Delta(E')\le \ell$ then $u$ and $v$ are still in the same connected component. We write $u\sim_{(\ell)}v$ in this case.

We shall also associate to a graph $G$ an auxiliary graph $H$, called
\emph{$k$-codeg graph of $G$}. $H$ is on the same vertex set as $G$. Two
distinct vertices $v_1, v_2\in V(H)$ are adjacent in $H$ if and only if
$|\neighbor_G(v_1)\cap \neighbor_G(v_2)|\ge k$.

The following lemma summarizes properties of the relation $\sim_{(\ell)}$, and
of $k$-codeg graphs.

\begin{lemma}\label{lem:robustlyadjacent} $ $
\begin{enumerate}[label=(\alph{*})]
  \item\label{it:RA1} The relation $\sim_{(\ell)}$ is an equivalence relation on $V(G)$.
  The equivalence classes of $\sim_{(\ell)}$ are called \emph{$\ell$-islands}.
  \item\label{it:RA2} Suppose that a vertex $v$ of $G$ has more than $\ell$
  neighbors in some $\ell$-island $L$. Then $v\in L$.
  \item\label{it:RA3} If $G$ is vertex-transitive then all $\ell$-islands induce
  mutually isomorphic, vertex-transitive graphs.
  \item\label{it:N1} If $G$ is vertex-transitive then the $k$-codeg graph $H$ of $G$ is vertex-transitive as well. We have $\deg(H)\ge \frac{\deg(G)^2}n-k$.
  \item\label{it:RA4} Suppose that $n\ge 10\alpha^{-2}$. If $G$ is a vertex-transitive graph on $n$ vertices with valency at least $\alpha n$ then each  $(\alpha^2 n/5)$-island contains at least $\alpha^2 n/2$ vertices.
\end{enumerate}
\end{lemma} 

\begin{proof}
Parts~\ref{it:RA1}--\ref{it:RA2} are trivial. For part~\ref{it:RA3}, note that each automorphism of $G$ maps an $\ell$-island again onto an $\ell$-island. In particular, all $\ell$-islands induce mutually isomorphic graphs. Moreover, taking the set
$A\subset \Aut(G)$ of automorphisms of $G$ which map a given $\ell$-island $L$ onto
itself and considering the restriction $A_{|L}:=\{g_{|L}\::\:g\in A\}$ on $L$, we get a subgroup $A_{|L}\le \Aut(G[L])$ which witnesses vertex-transitivity of $G[L]$.

The first part of~\ref{it:N1} is obvious. For the second part we count the number of triples $(x,y,z)$ with $z$ adjacent to both $x$ and $y$ in two different ways to get
\begin{align*}
n\deg(G)^2&=\sum_{x,y\in V(G)}|\neighbor_G(x)\cap\neighbor_G(y)|\\
&\le \sum_{x,y\in
V(G),xy\in E(H)} (n-2)+ \sum_{x\in V(G)}(n-1)+\sum_{x,y\in V(G),x\neq y,xy\not\in E(H)}(k-1)\\
&=n(n-2)\deg(H)+n(n-1)+n(n-1-\deg(H))(k-1) \\
&\leqslant n^2 \deg(H) + n^2k\;,
\end{align*}
and the claim follows. 	

To prove Part~\ref{it:RA4}, consider the $(\alpha^2 n/2)$-codeg graph $H$ of $G$.
 By Part~\ref{it:N1}, $H$ is vertex-transitive of valency $\deg(H) \geqslant
\alpha^2 n/2$. Observe now that if
$|\neighbor_G(u)\cap\neighbor_G(v)|\ge2\frac{\alpha^2n}5+1$ then $u$ and $v$ lie
in the same $(\alpha^2 n/5)$-island; in particular, the conclusion applies when
$uv$ is an edge of $H$. Since $\deg(H) \geqslant \alpha^2 n/2$ we deduce that
each $(\alpha^2 n/5)$-island of $G$ contains at least $\alpha^2 n/2$ vertices.
\end{proof}

As a corollary of Lemma~\ref{lem:robustlyadjacent} we get the following.

\begin{lemma}\label{lem:IFnot}
Suppose $G$ is a vertex-transitive graph on $n$ vertices with valency at least $\alpha n$. If $G$ is not $(\alpha^4 n/40)$-robust, then there exists a partition $V(G)=V_1\dcup\ldots \dcup V_r$ with $2 \leqslant r \leqslant \frac2{\alpha^2}$ such that all the graphs $G[V_i]$ are isomorphic to the same vertex-transitive graph $G'$ of order  $n'$ and valency at least $4\alpha n'/3$.
\end{lemma}

\begin{proof}
Let $V_1\dcup \ldots\dcup V_r$ be the $(\alpha^4 n/40)$-islands of $G$. If $r=1$ then $G$ is $(\alpha^4 n/40)$-robust and there is nothing to prove. Thus we assume that $r>1$.

Observe that since $\alpha^4/40 < \alpha^2/5$, each $(\alpha^4 n/40)$-island consists of several $(\alpha^2 n/5)$-islands. In conjunction with Part~\ref{it:RA4} of Lemma~\ref{lem:robustlyadjacent}, we get that $r \leqslant 2\alpha^{-2}$. By Part~\ref{it:RA2} of Lemma~\ref{lem:robustlyadjacent} each vertex $v\in V_1$ sends at most $\alpha^4 n/40$ edges to $V_i$ for $i\neq 1$. It follows that
\[
\deg(v,V_1)\ge\alpha n-(r-1)\frac{\alpha^4 n}{40} \ge \alpha n - \frac{\alpha^2 n}{20} \geqslant \frac{2\alpha n}3\;.
\]
On the other hand, for $n'=|V_1|$ we have $n'=\frac{n}{r}\le \frac n2$. Therefore the valency of the graph $G'=G[V_1]$ is at least $4\alpha n'/3$. This proves the lemma.
\end{proof}

Lemma~\ref{lem:IFnot} says that if $G$ is not robust then we can partition it into a few island each having higher (by a constant factor) density than $G$. Repeating this process, it will follow that every dense vertex-transitive graph can be partitioned in a symmetric way into a bounded number of robust graphs.

\begin{lemma}\label{lem:goodIslands-robust}
For every $\alpha>0$ there exist numbers $R,N_0$ and $\mu\in(0,\alpha/2)$ such that the following holds: Suppose $G$ is a vertex-transitive graph of order $n>N_0$ and valency at least $\alpha n$. Then there exists a partition $V(G)=V_1 \dcup \cdots \dcup V_r$, into $r<R$ parts such that all the graphs $G[V_i]$ are isomorphic to a graph $G'$ which is vertex-transitive and $(\mu n)$-robust. Furthermore, for each $g \in \Aut(G)$ and each $1 \leqslant j \leqslant r$ we have $g(V_j) \in \{V_1,\ldots,V_r\}$. 
\end{lemma}
\begin{proof}
We first set up necessary constants. Let $Q=\lceil\log_{4/3}(\frac1\alpha)\rceil$, and $\alpha_i=(4/3)^i\alpha$ for $i=0,1,\ldots$. Let $R=\prod_{i=0}^Q(2\alpha_i^{-2})$, and $\mu=\alpha^4/(40R)$. Last, let $N_0$ be sufficiently large.

Set $G_0=G$, and $n_0=n$. Inductively, in steps $i=0,1,\ldots$ we either get that $G_i$ is $(\alpha_i^4n_i/40)$-robust, or by Lemma~\ref{lem:IFnot} that there is a partition $V(G_i)=V_{i,1}\dcup\ldots\dcup V_{i,r_i}$ (with $r_i\le 2/\alpha_i^2$) such that each graph $G_i[V_{i,j}]$ ($j=1,\ldots,r_i$) is isomorphic to a vertex-transitive graph $G_{i+1}$ of order $n_{i+1}$, thus allowing a next step of the iteration. By induction, and the properties of the partition output by Lemma~\ref{lem:IFnot} the vertex set of the original graph $G$ can be partitioned into vertex-sets inducing graphs isomorphic to $G_{i+1}$. Observe that it is guaranteed by Lemma~\ref{lem:IFnot} and induction that $G_{i+1}$ has valency at least $\alpha_{i+1}n_{i+1}$.

Since $\alpha_Q\ge 1$, the above procedure must terminate in step $i_\mathrm{stop}< Q$. It is easily checked that the partition of $V(G)$ into copies of $G_{i_\mathrm{stop}}$ satisfies the assertions of the lemma.
\end{proof}

Observe that $\ell$-iron connectivity implies $\ell$-robustness. If the converse was true then we could immediately deduce Theorem~\ref{lem:goodIslands} from Lemma~\ref{lem:goodIslands-robust}. However, the converse is very far from being true. For example, the union of two cliques of size $2m$ having exactly one common vertex is $(m-1)$-robust but it is not even $1$-iron as the common vertex of the two cliques is a cut-vertex. The following lemma gives a partial converse for the class of vertex-transitive graphs.

\begin{lemma}\label{lem:RobustIron}
Let $G$ be a $(\mu n)$-robust vertex-transitive graph of order $n$ and valency at least $\alpha n$, for some $\alpha$, $\mu$, with $2\alpha/3>\mu>0$. Let $\lambda = \min\left\{\frac{\alpha}{2^{3 + 2/\alpha}}, \frac{\mu}{2^{2 + 2/\alpha}} \right\}$. Then $G$ is $(\lambda n)$-iron.
\end{lemma}

Before diving into the proof of Lemma~\ref{lem:RobustIron} let us give a
heuristic why the lemma ought to hold. The graph $G$ is robust by the
assumptions of the lemma. On the other hand it is
known~(\cite[Theorem~3.4.2]{Godsil:Book}) that connected
vertex-transitive graphs of high valency have high vertex connectivity. Therefore one can hope for a combination of the two properties, that is for iron connectivity.

\begin{proof}[Proof of Lemma~\ref{lem:RobustIron}]
Let $d\ge \alpha n$ be the valency of $G$. Suppose for contradiction that $G$ is not $(\lambda n)$-iron. That is, we have a partition $V(G)=A_0\dcup U_0\dcup B_0$, $|U_0|\le \lambda n$, $\Delta_G(A_0,B_0)\le \lambda n$. We proceed with an iterative procedure described below. For $i\ge 0$ we are given a partition $V(G)=A_i\dcup U_i\dcup B_i$. We further have the following properties:
\begin{enumerate}
  \item[$\textrm{(I1)}_i$] $|U_i|\le 2^i\lambda n$,
  \item[$\textrm{(I2)}_i$] $\Delta_G(A_i,B_i)\le 2^i\lambda n$, and
  \item[$\textrm{(I3)}_i$] $0<|A_i|\le n-\frac{i\alpha n}2$.
\end{enumerate}
We terminate this iterative procedure when for each $g\in \Aut(G)$, if there is
an $a\in A_i$ such that $g(a)\in A_i$ then for each $b\in B_i$ we have that
$g(b)\not\in A_i$. Otherwise, as we shall show below, we can produce a partition $V(G)=A_{i+1}\dcup U_{i+1}\dcup B_{i+1}$ satisfying $\textrm{(I1)}_{i+1}$, $\textrm{(I2)}_{i+1}$, and $\textrm{(I3)}_{i+1}$. Note that from $\textrm{(I3)}$ it follows that we must terminate in $i_\mathrm{stop}<\frac2\alpha$ steps.
  
Suppose we did not terminate in step $i$. Then there exists $g\in \Aut(G)$, $a\in A_i$, $b\in B_i$ such that $g(a), g(b)\in A_i$. Observe that $\textrm{(I2)}_i$ gives $|\neighbor(b)\setminus (B_i\cup U_i)|\le 2^i\lambda n$, and consequently with the help of $\textrm{(I1)}_i$ we have $|\neighbor(b)\setminus B_i|\le 2^{i+1}\lambda n$. Similarly, $|\neighbor(g(b))\setminus A_i|\le 2^{i+1}\lambda n$. We conclude that
\begin{equation}\label{eq:AigBi}
\begin{aligned}
|A_i\cap g(B_i)| &\ge |\neighbor(g(b))|-|\neighbor(g(b))\setminus A_i|-|\neighbor(g(b))\setminus g(B_i)|\\ 
                 &= d - |\neighbor(g(b))\setminus A_i|-|\neighbor(b)\setminus B_i| \\
                 &\ge \alpha n - 2^{i+2}\lambda n \ge \frac{\alpha n}2,
\end{aligned}
\end{equation}
where the last inequality follows since $\alpha \ge 2^{3 + 2/\alpha} \lambda \ge 2^{3 + i}\lambda$.

Define $A_{i+1}=A_i\cap g(A_i)$, $U_{i+1}=U_i\cup g(U_i)$, and $B_{i+1}=(B_i\cup
g(B_i))\setminus U_{i+1}$. This is a partition of $V(G)$ (see
Figure~\ref{fig:interectionsVT}).
\begin{figure}[ht]
    \centering
    \includegraphics{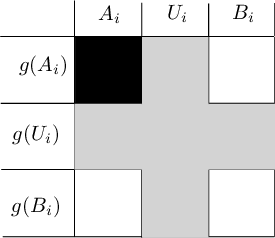}
     \caption{The sets $A_{i+1}$, $U_{i+1}$ and $B_{i+1}$ as
     intersections of the sets $A_i$, $U_i$, $B_i$, $g(A_i)$, $g(U_i)$, and
     $g(B_i)$. The set $A_{i+1}$ is represented by black, $U_{i+1}$ by grey, and $B_{i+1}$ by white.}
\label{fig:interectionsVT}
\end{figure}
$\textrm{(I1)}_{i+1}$ and $\textrm{(I2)}_{i+1}$ are obviously satisfied. The lower bound in $\textrm{(I3)}_{i+1}$ follows from the fact that $g(a)\in A_i\cap g(A_i)$. The upper bound is then established through the following chain of inequalities:
\[
|A_i\cap g(A_i)|\le |A_i|-|A_i\cap g(B_i)|\leByRef{eq:AigBi}|A_i|-\frac{\alpha n}2\;.
\]
This finishes the iterative step.

We now deal with the situation of termination in the step $i_\mathrm{stop}<\frac2\alpha$. For simplicity, we write $A=A_{i_\mathrm{stop}}$, $B=B_{i_\mathrm{stop}}$, and $U=U_{i_\mathrm{stop}}$. We have 
\begin{equation}\label{eq:WH}
|U|\le 2^{i_\mathrm{stop}} \lambda n < 2^{2/\alpha} \lambda n \leqslant \frac14\mu n \quad\mbox{and similarly}\quad \Delta_G(A,B)\le \frac14\mu n\;.
\end{equation}
Furthermore, we have
\begin{equation}\label{eq:star}
\text{For every $g\in \Aut(G)$, if $g(a')\in A$ for some $a'\in A$, then $g(b')\not\in A$ for each $b'\in B$.}
\end{equation}
We first prove that each vertex $u\in U$ has either almost all its neighbors in $A$, or in $B$.

\begin{AuxiliaryCl}\label{AC:splitU}
For each $u \in U$, either $|\neighbor(u)\cap A|\ge d-\frac34\mu n$, or $|\neighbor(u)\cap B|\ge d-\frac34\mu n$.
\end{AuxiliaryCl}
\begin{proof}[Proof of Claim~\ref{AC:splitU}]
As $d-\frac34\mu n>\frac d2$, we have that at most one of the assertions of the claim can hold for a given vertex $u\in U$. Suppose now the statement fails for some $u\in U$. Then we have
\begin{align}
|\neighbor(u)\cap A| = |\neighbor(u)| - |\neighbor(u) \cap B| - |\neighbor(u) \cap U| &>\frac\mu2 n\;, \mbox{and}\label{eq:NuA}\\
|\neighbor(u)\cap B| = |\neighbor(u)| - |\neighbor(u) \cap A| - |\neighbor(u) \cap U| &> \frac\mu2 n\;. \label{eq:NuB}
\end{align}
Let $a\in A$ be arbitrary and take a $g\in \Aut(G)$ such that $g(u)=a$. We then
have $\neighbor(a)= \neighbor(g(u)) = g(\neighbor(u))$, and in particular
$g(\neighbor(u)\cap A)\subset \neighbor(a)$.

We claim that there exists an $a'\in \neighbor(u)\cap A$ such that $g(a')\in A$. Indeed, if this was not the case, then $g(x)\in B\cup U$ for each $x\in\neighbor(u)\cap A$. Therefore, we would then have
\begin{align*}
|\neighbor(a)\cap (B\cup U)|&=|g(\neighbor(u))\cap (B\cup U)|\ge |(g(\neighbor(u) \cap A)\cap (B\cup U)|\\
&=|g(\neighbor(u)\cap A)| = |\neighbor(u) \cap A| \gByRef{eq:NuA} \mu n/2\;,
\end{align*}
contradicting~\eqref{eq:WH}.

Similarly, using~\eqref{eq:NuB} and the fact that $g(\neighbor(u)\cap B)\subset g(\neighbor(u)) = \neighbor(a)$, we get that there exists a $b'\in\neighbor(u)\cap B$ such that $g(b')\in A$. The properties of $g,a'$ and $b'$ contradict~\eqref{eq:star}.
\end{proof}

By Claim~\ref{AC:splitU} we have a partition $U=U_A\dcup U_B$, where $U_A=\{u\in U\::\: \deg(u,A)\ge d-\frac34\mu n\}$ and $U_B=\{u\in U\::\: \deg(u,B)\ge d-\frac34\mu n\}$. Define $V_1=A\cup U_A$ and $V_2=B\cup U_B$. We have $V_1,V_2\neq\emptyset$. It is straightforward to verify that $\Delta_G(V_1,V_2)\le \mu n$. This contradicts the fact that $G$ is $(\mu n)$-robust.
\end{proof}

Observe now that Lemma~\ref{lem:RobustIron} together with Lemma~\ref{lem:goodIslands-robust} immediately imply Theorem~\ref{lem:goodIslands}.

\smallskip
We conclude this section with three easy lemmas which are tailored for 
applications later in the proof of Theorem~\ref{thm:RobustOKbipfar}.
\begin{lemma}\label{lem:conK2}
Suppose that a graph $H$ is $\ell$-iron. Then the 2-blow-up $2\times H$ is also
$\ell$-iron.\footnote{In fact it is not much more difficult to show that the 2-blow-up is $2\ell$-iron but $\ell$-iron connectivity is enough for our purposes and has a clearer proof.}
\end{lemma}
\begin{proof}
Observe first, that the minimum degree of $H$ is at least $2\ell+1$. Indeed, if
there exists a vertex $v$ with $\deg(v)\le 2\ell$ then this vertex can be
isolated from the rest of the graph by deletion of at most $\ell$ edges incident
with $v$, and at most $\ell$ vertices in the neighbourhood of $v$.

Observe that there are two natural vertex disjoint copies of $H$ in $2 \times H$, say $H_1$ and $H_2$. Consider any sets $E'\subset E(2\times H)$,
with $\Delta(E')\le \ell$ and $V'\subset V(2\times H)$ with $|V'|\le \ell$. Since $H$ is $\ell$-iron, both $H_1$ and $H_2$ remain connected after the removal of $V'$ and $E'$. Since the minimum degree of $H$ is at least $2\ell+1$, then every vertex of $H_1$ has at least $2\ell+1$ neighbours in $H_2$. In particular after the removal of $V'$ and $E'$ there is still an edge between $H_1$ and $H_2$ and therefore $(2 \times H) \setminus (V'\cup E')$ is still connected. Therefore $2 \times H$ is $\ell$-iron.
\end{proof}

\begin{lemma}\label{lem:witNOTiron}
Let $R'$ be a graph on $k'$ vertices. Suppose that there exist sets
$L_1,L_2\subset V(R')$ such that $|L_1|\le \sqrt{\rho}k'$, and
$e(L_2,V(R')\setminus (L_1\cup L_2))\le \rho k'^2$. If there exists disjoint
sets $W_1, W_2\subset V(R')\setminus(L_1\cup L_2)$, such that $\neighbor(W_2) \subseteq L_1 \cup L_2$,
and $\min\{|W_1|,|W_2|\}>2\sqrt{\rho}k'$, then $R'$ is not
$(2\sqrt{\rho}k')$-iron.
\end{lemma}
\begin{proof}
Let $L=\{v\in L_2\::\: \deg(v,V(R')\setminus (L_1\cup L_2))\ge 2\sqrt{\rho}k'\}$, and $P=\{v\in V(R')\setminus (L_1\cup L_2)\::\:\deg(v,L_2)\ge 2\sqrt{\rho}k'\}$. We have $\max\{|L|,|P|\}\le \sqrt{\rho}k'/2$. In particular, 
\begin{equation}\label{eq:disc}
W_1\setminus (L_1\cup L\cup P)\neq\emptyset\quad\mbox{and}\quad W_2\setminus
(L_1\cup L\cup P)\neq\emptyset\;.
\end{equation} 
Define $E'\subset E(R')$ to be edges running between $L_2\setminus L$ and
$V(R')\setminus (L_1\cup L_2\cup P)$. We have $\Delta_{R'}(E')\le 2\sqrt{\rho}k'$. By~\eqref{eq:disc}, $R'$ is not connected after removal of the vertex set $L_1\cup L\cup P$ and the edge set $E'$. Indeed, after the removal of $E'$ we have that there are no more edges between $W_2\setminus
(L_1\cup L\cup P)$ and $V(R') \setminus (W_2 \cup L_1\cup L\cup P)$. Therefore, $R'$ is not $(2\sqrt{\rho}k')$-iron.
\end{proof}

\begin{lemma}\label{lem:minDegDiam}
Let $H$ be an $n$-vertex $h$-strongly connected digraph and let $x,y$ be two distinct vertices of $H$. Then there exists a (directed) path from $x$ to $y$ of length at most $\frac{n}{h}+1$.
\end{lemma}
\begin{proof}
By directed version of Menger's Theorem (cf.~\cite[Theorem~7.3.1(b)]{DigraphBook}), there exist $h$
internally vertex-disjoint directed paths from $x$ to $y$. Therefore one of these paths must contain at most $\frac{n-2}{h}$ internal vertices and so must have length at most $\frac{n-2}h+1 \leqslant \frac{n}{h}+1$.
\end{proof}


\section{Bipartite case}\label{sec:bipartite}
In this section we give a fine description of dense vertex-transitive graphs
which are almost bipartite. Their properties are stated in Lemma~\ref{lem:almostBipBalanced}.

The \emph{edit distance} $dist(G_1,G_2)$ between two $n$-vertex graph is the
number of edges one needs to edit (i.e.~to either remove or add) to get $G_2$ from $G_1$,
minimized over all identification of $V(G_1)$ with $V(G_2)$. Given an $n$-vertex
graph $G$, we say that it is \emph{$\epsilon$-close} to a graph property
$\mathcal P$ if there exists an $n$-vertex graph $H\in\mathcal P$ such that
$dist(G,H)<\epsilon n^2$. Otherwise we say that it is \emph{$\epsilon$-far} from
$\mathcal P$.

\begin{lemma}\label{lem:almostBipBalanced}
Let $c\in (0,\frac{1}{17})$ be arbitrary. Suppose that $G$ is a $cn$-iron
vertex-transitive graph $G$ on $n$ vertices which is $c^4$-close to bipartiteness. Then there exist a bipartition $V(G)=A\dcup B$ such that $|A|=|B|$, for each $u\in A$ and each $v\in B$ we have $\deg(u,A) \leqslant 6c^2n$, and $\deg(v,B) \leqslant 6c^2n$. Furthermore, we have $g(A)=A$ or $g(A)=B$ for each $g\in \Aut(G)$.
\end{lemma}

\begin{proof}
We write $\Delta$ for the valency of $G$. Observe that since $G$ is $cn$-robust,
then $\Delta \geqslant cn$. Let $A\dcup B=V(G)$ be the bipartition which
maximizes $e(A,B)$. We have
\begin{equation}\label{eq:TR}
e(A)+e(B)<c^4n^2\;.
\end{equation} We claim that
\begin{equation}\label{eq:min}\min\{|A|,|B|\}\ge \frac n3\;.\end{equation}
Indeed, suppose for contradiction that, for example, $|A|>\frac{2n}3$ and $|B|<\frac n3$. Counting $e(A,B)$ in two ways we arrive to $\sum_{v\in A}\deg(v)-2e(A)=\sum_{v\in B}\deg(v)-2e(B)$, and therefore 
\[
\frac{2 \Delta n}{3} < \Delta|A| \leqslant \Delta|B| + 2c^4n^2 < \frac{\Delta n}{3} + 2c^4n^2,
\]
a contradiction as $\Delta \geqslant cn$ and $c$ is sufficiently small. This
proves~\eqref{eq:min}.

Define $A'=\{v\in A\::\: \deg(v,A)\ge c^2n\}$, and $B'=\{v\in B\::\:
\deg(v,B)\ge c^2n\}$. By~\eqref{eq:TR} we have $|A'|+|B'|<2c^2n$. Together with~\eqref{eq:min} this gives that
\begin{equation}\label{eq:NE}
A\setminus A'\neq\emptyset\quad\mbox{and}\quad B\setminus B'\neq\emptyset\;.
\end{equation}

\begin{AuxiliaryCl}\label{AC:approxAuto}
For each $g\in\Aut(G)$ we either have $|A\cap g(A)|\ge |A|-5c^2n$ or $|A\cap g(B)|\ge |A|-5c^2n$.  Also, for each $g\in\Aut(G)$ we either have $|B\cap g(A)|\ge |B|-5c^2n$ or $|B\cap g(B)|\ge |B|-5c^2n$.
\end{AuxiliaryCl}

\begin{proof}[Proof of Claim~\ref{AC:approxAuto}]
It is enough to prove the first statement.

We start with some general calculations. We shall later use them to show that if $g\in\Aut(G)$ failed to fulfil the assertions we would
get a contradiction to $cn$-iron connectivity.

Let $\tilde A=A\setminus A'$ and $\tilde B=B\setminus B'$. Consider the partition $V(G)=X\dcup Y\dcup U$, where $X=(\tilde A\cap g(\tilde A))\cup (\tilde B\cap g(\tilde B))$, $Y=(\tilde A\cap g(\tilde B))\cup (\tilde B\cap g(\tilde A))$, and $U=V(G)\setminus (X\cup Y)$. We have 
\begin{equation}\label{eq:Us}
|U|\le |A'| + |B'| + |g(A')| + |g(B')| \leqslant 4c^2n\le cn\;.
\end{equation}
We claim that 
\begin{equation}\label{eq:MaxS}
\Delta_G(X,Y)\le cn\;.
\end{equation}
To prove this it suffices to prove that
\begin{equation}\label{eq:Maxs}
\max\left\{\Delta_{\tilde A\tilde A,\tilde A\tilde B},\Delta_{\tilde A\tilde A,\tilde B\tilde A},\Delta_{\tilde B\tilde B,\tilde A\tilde B},\Delta_{\tilde B\tilde B,\tilde B\tilde A},\Delta_{\tilde A\tilde B,\tilde A\tilde A},\Delta_{\tilde B\tilde A,\tilde A\tilde A},\Delta_{\tilde A\tilde B,\tilde B\tilde B},\Delta_{\tilde B\tilde A,\tilde B\tilde B} \right\}\le \frac{cn}2\;,
\end{equation}
where $\Delta_{CD,EF}=\max\{\deg(v,E\cap g(F)):v\in C\cap g(D)\}$ defines the eight new symbols above. Here we only bound the first two terms; the methods to this end apply to the remaining six as well. To prove that $\Delta_{\tilde A\tilde A,\tilde A\tilde B}\le \frac{cn}2$, consider an arbitrary $v\in \tilde A\cap g(\tilde A)$. We have $v\not\in A'$. We then have $\deg(v,\tilde A\cap g(\tilde B))\le\deg(v,\tilde A)\le \deg(v,A)< c^2n$, where the last inequality follows from the definition of the set $A'$. To bound $\Delta_{\tilde A\tilde A,\tilde B\tilde A}$ we again consider an arbitrary $v\in \tilde A\cap g(\tilde A)$. We have 
$$\deg(v,\tilde B\cap g(\tilde A))=\deg(g^{-1}(v),g^{-1}(\tilde B)\cap g^{-1}(g(\tilde A)))=\deg(g^{-1}(v),g^{-1}(\tilde B)\cap \tilde A)\le \deg(g^{-1}(v),\tilde A)\;.$$
We observe that $g^{-1}(v)\in g^{-1}(\tilde A)\cap g^{-1}(g(\tilde A))\subset \tilde A$, and the bound follows by the definition of the set~$A'$.

\medskip

Suppose now that the statement of the Claim fails for $g\in\Aut(G)$. We then
have $X\neq \emptyset$ and $Y\neq \emptyset$. Indeed, to show for example
that $X\neq\emptyset$, we note that $$|X|\ge |A\cap
g(A)|-|A'|-|g(A')|>5c^2n-2c^2n-2c^2n>0\;.$$Let $E'$ be the edges of $G$ running
between $X$ and $Y$. Now if we remove $U$ and $E'$ from $G$ we get a disconnected graph. Together with the bounds~\eqref{eq:Us} and~\eqref{eq:MaxS} this proves that $G$ is not $cn$-iron, a contradiction.
\end{proof}

\begin{AuxiliaryCl}\label{AC:EveryVertexSmall}
For every $v\in A$ we have $\deg(v,A)\le 6c^2n$. Also, for every $v\in B$ we have $\deg(v,B)\le 6c^2n$.
\end{AuxiliaryCl}

\begin{proof}[Proof of Claim~\ref{AC:EveryVertexSmall}]
By symmetry, it suffices to prove the first part of the statement. Let $w\in
B\setminus B'$ be arbitrary; such a choice is possible by~\eqref{eq:NE}. Let $v \in
A$ and take $g\in \Aut(G)$ be such that $g(v)=w$. Let $P=\neighbor(v)\cap A$,
and $Q=\neighbor(v)\cap B$. Suppose for contradiction  that $|P|>6c^2n$. Since
the bipartition $A\dcup B$ was chosen to maximize $e(A,B)$, we must have
$|Q|\ge\frac{cn}2$. Since $\neighbor(w)=g(P)\cup g(Q)$ and since also $w \not\in B'$ we have that $|g(A) \cap A| \geqslant |g(P) \cap A| > 5c^2 n$ and so $|g(A) \cap B| < |B| - 5c^2 n$. Similarly, we also have $|g(B) \cap A| \geqslant |g(Q) \cap A| > 5c^2 n$ and so $|g(B) \cap B| < |B| - 5c^2 n$. But these contradict Claim~\ref{AC:approxAuto}.
\end{proof}

\begin{AuxiliaryCl}\label{AC:AllOrNothing}
For every $g\in \Aut(G)$ we either have $A\cap g(A)=\emptyset$, or $A\cap g(B)=\emptyset$. Likewise, we have $B\cap g(A)=\emptyset$, or $B\cap g(B)=\emptyset$.
\end{AuxiliaryCl}

\begin{proof}[Proof of Claim~\ref{AC:AllOrNothing}]
Let $C, D \in \{A,B\}$. Let $C'=V(G)\setminus C$, and $D'=V(G)\setminus D$. (Thus $C', D' \in  \{A,B\}$.)

Suppose that $C\cap g(D)\neq \emptyset$. We can take a $v\in C$ with $g^{-1}(v)\in D$.
Using Claim~\ref{AC:EveryVertexSmall} for $g^{-1}(v)$, and then for $v$ we get.
\begin{align*}
6c^2 n \ge &\deg(g^{-1}(v), D) = |\neighbor(g^{-1}(v)) \cap D| = |\neighbor(v)
\cap g(D)| \ge |\neighbor(v)\cap C' \cap g(D)| \\
 &= |\neighbor(v)\cap C'| -
|\neighbor(v)\cap C' \cap g(D')| \ge |\neighbor(v)| - |\neighbor(v) \cap C| -
|C' \cap g(D')| \\ &\ge \Delta - 6c^2n - |C' \cap g(D')| \ge cn - 6c^2 n - |C' \cap g(D')|.
\end{align*}
Thus $|C'\cap g(D')| \ge cn -12c^2 n > 5 c^2 n$. Hence $C' \cap g(D') \neq \emptyset$. Repeating the previous argument for $C'$ and $D'$ yields $|C \cap g(D)| > 5 c^2 n$.

Therefore for every $C, D \in \{A,B\}$ we have 
\begin{equation}\label{eq:Percy}
|C\cap g(D)|=0\text{ or }|C\cap g(D)|>5c^2n\;.
\end{equation}
We use this for $C=A$ and $D=B$. We get that $|A\cap g(B)|=0$, or $|A\cap g(B)|>5c^2n$. We are done in the former case. In the latter case, we have $|A\cap g(A)|<|A|-5c^2n$. Claim~\ref{AC:approxAuto}  then gives that $|A\cap g(B)|>|A|-5c^2n$. Using again~\eqref{eq:Percy}, this time with $C=A$, $D=A$, we get that $|A\cap g(A)|=0$.
\end{proof}

Claims~\ref{AC:EveryVertexSmall} and~\ref{AC:AllOrNothing} show that the
bipartition $A\dcup B$ satisfies the conclusion of Lemma~\ref{lem:almostBipBalanced}.
\end{proof}

\begin{remark}\label{rem:17}
In the above proof we showed that the partition maximizing $e(A,B)$ satisfies
the conclusion of Lemma~\ref{lem:almostBipBalanced}. In fact we only used the following two properties of the partition:
\begin{enumerate}
\item The partition satisfies~\eqref{eq:TR}.
\item For every $v \in A$ we have $\deg(v,A) \leqslant \deg(v,B)$ and for every
$v\in B$ we have that $\deg(v,B) \leqslant \deg(v,A)$.
\end{enumerate} 
In particular any partition satisfying the above  two properties also satisfies
the conclusion of Lemma~\ref{lem:almostBipBalanced}.  This fact will be
important in the proof of Theorem~\ref{thm:algo} which provides an algorithmic
version of Theorem~\ref{thm:main}.
\end{remark}

\begin{remark}\label{rem:bipartiteISvt}
Note that the bipartite subgraph $G[A,B]$ obtained from the partition $A\dcup B$ given by Lemma~\ref{lem:almostBipBalanced} by removing all edges within the parts $A$ and $B$ is itself vertex-transitive. Indeed observe that for any automorphism $g \in \Aut(G)$ and any edge $e$ between the parts $A$ and $B$ we have that $g(e)$ also lies between these parts. Therefore every automorphism of $G$ restricted to $G[A,B]$ is also an automorphism and so $G[A,B]$ is vertex-transitive. 
\end{remark}


\section{\Szemeredi's Regularity Lemma}\label{sec:regularity}

\Szemeredi's Regularity Lemma is one of the main tools in our proof of Theorem~\ref{thm:main}. In this section we collect all the tools related to the Regularity Lemma that we will need. For surveys on the Regularity Lemma and its applications we refer the reader to~\cite{KS96,KSS_bl,KSSS00,KuhnOsthusSurv}.

Before stating the lemma, we need to introduce some more notation. The {\em
density} of a bipartite graph $G$ with vertex classes $A$ and $B$ is defined to be $d_G(A,B) = \frac{e(A,B)}{|A||B|}$. We sometimes write $d(A,B)$ for $d_G(A,B)$ if this is unambiguous. Given $\eps > 0$, we say that $G$ is $\eps$-{\em regular} if for all subsets $X \subseteq A$ and $Y \subseteq B$ with $|X| \geq \eps|A|$ and $|Y| \geq \eps |B|$ we have that $|d(X,Y) - d(A,B)| < \eps$. Given $d \in [0,1]$, we say that $G$ is $(\eps,d)$-{\em regular} if it is $\eps$-regular of density at least $d$. We also say that $G$ is $(\eps,d)$-{\em super-regular} if it is $\eps$-regular and furthermore $d_G(a) \geq d|B|$ for all $a \in A$ and $d_G(b) \geq d|A|$ for all $b \in B$. Given partitions $V_0\dcup V_1\dcup \ldots\dcup V_k$ and $U_1\dcup \ldots\dcup U_{\ell}$ of the vertex set of some graph, we say that $V_0\dcup V_1\dcup \ldots\dcup V_k$ \emph{refines} $U_1\dcup \ldots\dcup U_{\ell}$ if for all $i$ with $1 \leqslant i \leqslant k$, there is some $1 \leqslant j \leqslant \ell$ with $V_i \subseteq U_j$. Note that this is weaker than the usual notion of refinement as we do not require $V_0$ to be contained in any $U_j$. We will use the following degree form of \Szemeredi's Regularity Lemma~\cite{Sze78}:

\begin{lemma}[Regularity Lemma; Degree form]\label{lem:SRL}
Given $\eps \in (0,1)$ and  integers $N',\ell$, there are integers $N = N(\eps,N',\ell)$ and $n_0 = n_0(\eps,N',\ell)$ such that if $G$ is any graph on $n \geq n_0$ vertices, $d \in [0,1]$ is any real number, and $U_1,\ldots,U_{\ell}$ is any partition of the vertex set of $G$, then there is a partition of the vertex set of $G$ into $k+1$ classes $V_0\dcup V_1\dcup \ldots\dcup V_k$, and a spanning subgraph $G'$ of $G$ with the following properties:
\begin{itemize}
\item[(i)] $N' \leq k \leq N$;
\item[(ii)] $V_0\dcup V_1\dcup \ldots \dcup V_k$ refines $U_1\dcup \ldots\dcup U_{\ell}$;
\item[(iii)] $|V_0| \leq \eps n, |V_1| = \cdots = |V_k| = m$;
\item[(iv)] $\deg_{G'}(v) \geq \deg_G(v) - (d + \eps)n$ for every $v \in V(G)\setminus V_0$;
\item[(v)] $G'[V_i]$ is empty for every $0 \leq i \leq k$, and no edges of $G'$ are incident with $V_0$; 
\item[(vi)] all pairs $(V_i,V_j)$ with $1 \leq i < j \leq k$ are
$\eps$-regular with density either 0 or at least $d$.
\end{itemize}
\end{lemma}

We call $V_1, \ldots, V_k$ the {\em clusters} of the partition, $V_0$ the {\em exceptional set} and the vertices of $G$ in $V_0$ the {\em exceptional vertices}. The {\em reduced graph} $R = R_{G'}$ of $G$ with respect to the above partition and the parameters $\eps$ and $d$ is the graph whose vertices are the clusters $V_1 \ldots, V_k$ in which $V_iV_j$ is an edge precisely when the pair $(V_i,V_j)$ has density at least $d$ in $G'$.

\begin{remark}\label{rem:2densities}
It turns out that for the proofs of Theorems~\ref{thm:RobustOKbipfar} and~\ref{thm:RobustOKbip} (see below) we need to work with two threshold densities $d_1 < d_2$ of the reduced graph. The degree form of the Regularity Lemma can be adapted in order to accommodate this need. In particular we can get a partition $V_0\dcup V_1\dcup \ldots\dcup V_k$ of the vertex set of $G$ and spanning subgraphs $G_1,G_2$ of $G$ such that properties (i)--(vi) of the Regularity Lemma hold for both $G_1$ and $G_2$ with the corresponding densities $d_1$ and $d_2$. (This can be deduced in the same way as the degree form of the Regularity Lemma is deduced from the standard form.)
\end{remark}

For further use, we also recall the following well-known facts. The next lemma says that large sub-pairs of regular pairs are regular.

\begin{lemma}\label{lem:Slicing}
Let $(A,B)$ be an $(\eps,d)$-regular pair with $\eps \leqslant d/2$ and let $A'$ and $B'$ be subsets of $A$ and $B$ of sizes $|A'| \geqslant |A|/3$ and $|B'| \geqslant |B|/3$. Then $(A',B')$ is $(3\eps,d/2)$-regular.
\end{lemma}

Given any bounded degree subgraph $H$ of the reduced graph $R$ we can make the pairs corresponding to its edges super-regular by removing a small fraction of the vertices of each cluster to the exceptional set. We will only need this fact in the case that $H$ is a matching.

\begin{lemma}\label{lem:Super-regular} 
Suppose $0 < 4\eps < d \leqslant 1$ and let $V_0\dcup V_1\dcup \ldots\dcup V_k$ be a partition of a graph $G$ as given by the Regularity Lemma. Let $R$ be the reduced graph with respect to this partition and the parameters $\eps$ and $d$. Let $M$ be a matching in $R$. Then we can move exactly $\eps m$ vertices from each cluster $V_i$ into $V_0$ such that each pair of clusters corresponding to an edge of $M$ is $(2\eps,d/2)$-super-regular while each pair of clusters corresponding to an edge of $R$ is $(2\eps,d/2)$-regular. 
\end{lemma}

Given an $(\eps,d)$-super-regular pair $(A,B)$, we will often need
to isolate a small sub-pair that maintains super-regularity in any
sub-pair that contains it. For $A^* \subseteq A$ and $B^* \subseteq B$
we say that $(A^*,B^*)$ is an \emph{$(\eps^*,d^*)$-ideal for $(A,B)$}
if for any $A^* \subseteq A' \subseteq A$ and $B^* \subseteq B' \subseteq B$
the pair $(A',B')$ is $(\eps^*,d^*)$-super-regular.
The following lemma shows that ideals exist.

\begin{lemma}[{\cite[Lemma~15]{ChKuKeeOst:Semiexact}}]\label{lem:ideal}
Suppose $0<\eps \ll \theta,d < 1/2$, and let $(A,B)$ be an $(\eps,d)$-super-regular pair with $|A| = |B| = m$, where $m$ is sufficiently large. Then there exists subsets $A^* \subseteq A$ and $B^* \subseteq B$ of sizes $\theta m$ such that $(A^*,B^*)$ is an $(\eps/\theta,\theta d/4)$-ideal for $(A,B)$.
\end{lemma}

The proof of the above lemma given in~\cite{ChKuKeeOst:Semiexact} is
probabilistic. (It proves that random subsets of sizes $\theta m$ have the
required property with high probability.) For finding the Hamilton cycle
efficiently in Theorem~\ref{thm:algo} below we will also need a
`constructive' proof of this lemma. We proceed to give such a proof.

\begin{proof}[Proof of Lemma~\ref{lem:ideal}] By using a more general version of
the Lemma~\ref{lem:Slicing}, it is enough to construct subsets $A^* \subseteq A$
and $B^* \subseteq B$ of sizes $\theta m$ such that every vertex $a \in A$ has
$\deg(a,B^*) \geqslant \theta dm/4$ and every vertex $b \in B$ has $\deg(b,A^*)
\geqslant \theta dm/4$. By symmetry, it is enough to show how to construct a
subset $A^* \subseteq A$ of size $\theta m$ such that very vertex $b \in B$ has
$\deg(b,A^*) \geqslant \theta dm/4$. We will construct this set $A^*$ by adding
to it one vertex at every step. At each step we will say that a vertex $b$ of $B$ is
\emph{unhappy} if it has $k < \theta dm/4$ neighbours in $A^*$. If a vertex $b$
is unhappy we will define its unhappiness $u(b)$ to be
$u(b)=\sum_{r=k+1}^{\theta dm/4} 2^{-r}$. Otherwise we define its unhappiness
$u(b)$ to be equal to 0. We also denote by $U$ the total unhappiness $U =
\sum_{b \in B}u(b)$ of vertices of $B$. Observe that if in the next step we add to $A^*$ a neighbour of $b$ then the unhappiness
of $b$ is reduced by at least $u(b)/2$. Note also that if a vertex $b$ is
unhappy, then it has at least $dm - \theta dm/4 \geqslant dm/2$ neighbours
outside of $A^*$. We now give to every edge joining $b$ to a vertex of $A
\setminus A^*$ a weight equal to $u(b)/2$. Then the total weight on these edges
is at least $\sum_{b\in B} u(b) dm/4 = Udm/4$. In particular there is a vertex $a
\in A \setminus A^*$ where the total weight on its incident edges is at least
$Ud/4$. Adding this vertex to $A^*$ we get that the new total unhappiness is at
most $(1-d/4)U$. Initially the total unhappiness was at most $m$. So after
$\theta m$ steps the total unhappiness is at most $(1-d/4)^{\theta m}m \leqslant
me^{-\theta m d/4} < 2^{-\theta dm/4}$, when $m$ is sufficiently large. But no
unhappy vertex can have unhappiness less than $2^{-\theta dm/4}$. It follows
that after $\theta m$ steps there is no unhappy vertex in $B$, as required.
\end{proof}

We will also need the following `blow-up'-type statement.
\begin{lemma}\label{lem:blow-up}
Suppose $0 < \eps \ll d$ and let $(A,B)$ be an $(\eps,d)$-super-regular pair with $|A| = |B|$. Let $a \in A$ and $b \in B$. Then $A \cup B$ contains a Hamilton path with endvertices $a$ and $b$.
\end{lemma}
\begin{proof}
The lemma follows from the Blow-up Lemma of Koml\'os, S\'ark\"ozy and Szemer\'edi~\cite{KSS_bl}. We need to deal with one minor difficulty which does not allow a direct application of the Blow-up Lemma, namely that we are prescribing exactly the images $a$ and $b$ of the endvertices of the Hamilton path.

Recall that by~\cite[Remark~13]{KSS_bl} we can impose additional restriction on a small number of target sets of vertices of the graph we are trying to embed in the super-regular pair. We thus proceed as follows.

We can assume that $|A|$ is sufficiently large. Otherwise, setting $\epsilon$ small, we can force $(A,B)$ to form a complete bipartite graph, and then the statement is trivial.

Let $A'\subset A$ and $B'\subset B$ be the neighbourhood of $b$ and $a$, respectively. We have $|A'\setminus\{a\}|\ge \frac{d|A|}2$, and $|B'\setminus\{b\}|\ge \frac{d|B|}2$. Observe also, that the pair $(A\setminus\{a\},B\setminus\{b\})$ is $(2\eps,\frac d2)$-super-regular. By the Blow-up Lemma we can find a Hamilton path $P$ in the pair $(A\setminus\{a\},B\setminus\{b\})$. Furthermore, by~\cite[Remark~13]{KSS_bl} we can require the endvertices of the path to lie in the sets $A'$ and $B'$. The path $aPb$ is a Hamilton path in $(A,B)$ satisfying the assertions of the lemma.
\end{proof}



\section{Hamilton cycles in iron connected vertex-transitive
graphs}\label{sec:HCinIronConnected}

In this section, we prove a stronger version of~Theorem~\ref{thm:main} under the additional assumption of high iron connectivity of the host graph. This is stated in Theorem~\ref{thm:RobustOKbipfar} in the non-bipartite setting, and in Theorem~\ref{thm:RobustOKbip} in the bipartite setting.

The basic idea is to follow \L uczak's `connected matching
argument'~\cite{Luczak:RCnCnCn}. The novel ingredient in our work is an innocent
looking modification of this technique: we observe that we can extend the
argument to work with fractional matchings as well. This allows one to use the
LP-duality. We believe that this observation will find further important
applications in the future. (After the first version of this manuscript was
posted on the arXiv, we learned that R\"odl and Ruci\'nski announced a solution
of a certain Dirac-type problem for hypergraphs using Farkas' Lemma, an approach
similar to our linear programming approach. The corresponding paper was posted later in the arXiv~\cite{RoRu:FarkasA}.) The use of the LP-duality in conjunction with the Regularity Lemma originated in discussion of Jan Hladk\'y with Dan Kr\'al' and
Diana Piguet. (As it was pointed to us by Deryk Osthus, the full strength of the LP-duality machinery is not needed. In~\cite{KuOs:Packings} it is shown that every dense almost regular graph has a reduced graph with an almost perfect matching and this suffices in our setting.)

\begin{theorem}\label{thm:RobustOKbipfar}
For every $\beta,\gamma>0$ and every $C \in \mathbb N$, there exists an $N_1$ such that every $\beta n$-iron vertex-transitive graph of order $n \geqslant N_1$ which is $\beta$-far from bipartiteness is $C$-pathitionable with an exceptional set
$U\subset V(G)$ with $|U|<\gamma n$.
\end{theorem}

\addtocounter{bipfar}{\value{theorem}}

\begin{theorem}\label{thm:RobustOKbip}
For every $c\in (0,\frac1{17})$, $\gamma > 0$ and $C\in\mathbb N$ there an exists $N_2$ such that for every vertex-transitive graph $G$ of order $n\ge N_2$ the following holds. Suppose $G$ is $cn$-iron  and $c^4$-close to bipartiteness. Let $A\dcup B$ be the bipartition of $G$ given by Lemma~\ref{lem:almostBipBalanced}. Then there exists a set $U\subset V(G)$ with $|U|<\gamma n$ such that $G$ is $C$-bipathitionable with exceptional set $U$ with respect to the partition $A\dcup B$.
\end{theorem}

After proving Theorem~\ref{thm:RobustOKbipfar} in detail below, we indicate necessary changes to make an analogous proof of Theorem~\ref{thm:RobustOKbip} work as well.

\begin{proof}[Proof of Theorem~\ref{thm:RobustOKbipfar}]
We begin by fixing additional constants $\eps,d_1,d_2,\gamma_1,\gamma_2$
satisfying
\[ 
0 < \eps \ll d_1 \ll \gamma_1 \ll \gamma_2 \ll d_2 \ll \gamma,\beta.
\]

Let $N' = 1/\eps$. Let $N(\eps,N',1)$ and $n_0(\eps,N',1)$ be the
numbers given by the Regularity Lemma. Set 
\begin{equation}\label{eq:n0Set}
n_0=\max\left\{\frac{N(\eps,N',1)}{\gamma_1},n_0(\eps,N',1)\right\}\;.
\end{equation}
Let $G$ be any $\beta n$-iron
connected vertex-transitive graph on $n \geqslant n_0$ vertices of valency $\Delta$. Apply the Regularity Lemma (see also Remark~\ref{rem:2densities}) with parameters
$\eps,N',\ell=1$ and $d_1,d_2$ to $G$ to obtain a partition $V_0\dcup V_1\dcup \ldots\dcup V_k$ of
$V(G)$. Let $G_1,G_2\subset G$ be the spanning subgraphs of $G$ given by the
Regularity Lemma corresponding to the densities $d_1$ and $d_2$ respectively. Let also $R_1$ and $R_2$ be the reduced graphs of $G$ with respect
to the above partition, the parameters $\eps$ and $d_1,d_2$ and the subgraphs $G_1$ and $G_2$ respectively. We write $m=|V_1|$.

We first claim that $R_1$ has a large fractional matching.
\begin{AuxiliaryBipfar}\label{AC:RBigM}
$\nu^*(R_1) \geqslant (1 - \frac{\gamma_1}2)\frac k2$. 
\end{AuxiliaryBipfar}
\begin{proof}[Proof of Claim~\ref{AC:RBigM}]
Observe that by Lemma~\ref{lem:CoverVertexTr} we have that $\fvc(G) = n/2$. We also have that 
\[
e(G_1) \geqslant e(G) - (d_1 + \eps)n^2\ge
\left(1-\frac{\gamma_1}{2}\right)e(G)\;,
\] 
where in the first inequality we used properties~(iii)--(v) of the Regularity Lemma and
in the second one we used the fact that $e(G)\ge \beta n^2/2$. By
Lemma~\ref{lem:coverAfter} we obtain that $\fvc(G_1) \geqslant (1 -
\frac{\gamma_1}{2})\frac{n}{2}$. Observe that
$\nu^*(G_1)=\nu^*(G_1-V_0)$ by property~(v) of the Regularity
Lemma. Therefore, combining Lemma~\ref{lem:LargeMatchingInherited} with
Theorem~\ref{thm:HalfIntegralMatching}\ref{it:Ma} we have
\[\nu^*(R_1)\ge\frac{\nu^*(G_1)}m=\frac {\fvc(G_1)}m\ge
\left(1-\frac{\gamma_1}{2}\right)\frac{n}{2 m}\ge
\left(1-\frac{\gamma_1}{2}\right)\frac k2\;. \qedhere\]
\end{proof}

\smallskip

The density $d_1$ was used to find a large matching in $R_1$ (cf.
Claim~\ref{AC:RBigM}). On the other hand, it is more convenient to work with
the higher threshold density $d_2$ to infer some connectivity properties of certain graphs that will be derived from $R_2$ (most importantly, to deduce Claim~\ref{AC:R^* highly
connected}).


Since $G$ is $\beta n$-iron and $\eps,d_2 \ll \beta$, properties (iii) and (iv)
of the Regularity Lemma (for the density $d_2$) imply that $G_2[V \setminus V_0]$
is $(\beta n/2)$-iron. We claim that the iron connectivity is inherited by the reduced graph $R_2$ as well.
\begin{AuxiliaryBipfar}\label{AC:Riron}
$R_2$ is $(\beta k/2)$-iron.
\end{AuxiliaryBipfar} 
\begin{proof}[Proof of Claim~\ref{AC:Riron}] Indeed, suppose we could disconnect
$R_2$ by removing a set of clusters $S$ of size at most $\beta k/2$ together
with an edge set $F\subset E(R_2)$ with $\Delta(F)\le\beta k/2$. Let $E'\subset
E(G_2[V \setminus V_0])$ be the set of edges contained in the regular pairs corresponding to $F$.
Then we could also disconnect $G_2[V \setminus V_0]$ by removing all vertices belonging to the clusters of $S$ together with the edge set $E'$. However, the clusters of $S$ contain at most $\beta k m/2 \leqslant \beta n/2$ vertices and also $\Delta(E')\le \beta k m/2 \leqslant
\beta n/2$. This would contradict the $(\beta n/2)$-iron connectivity of $G_2[V
\setminus V_0]$.
\end{proof}

For each $1 \leqslant i \leqslant k$, we arbitrarily partition $V_i$ into two
parts $V_i^1$ and $V_i^2$ of equal sizes. In the case that the $V_i$'s have odd
sizes, then before the partitioning we move an arbitrary vertex from each cluster
into $V_0$. We denote the new exceptional set obtained by $V_0'$. We also define
a new graph $R'_1$ on vertex set $\{V_1^1,V_1^2,\ldots,V_k^1,V_k^2\}$ where
$V_i^{s}$ is adjacent to $V_j^t$ if and only if $V_i$ was adjacent to $V_j$ in
$R_1$. Similarly, we define a graph $R'_2$ on the same vertex set as $R'_1$ to
be the 2-blow-up of $R_2$. By Lemma~\ref{lem:Slicing} every edge of $R'_1$ corresponds
to a $(3\eps,d_1/2)$-regular pair, and every edge of $R'_2$ corresponds
to a $(3\eps,d_2/2)$-regular pair. We have  $R_1'=2\times R_1$,
$R_2'=2\times R_2$, and $R'_2\subset R'_1$. Consider a matching $M$ in $R'_1$ of
weight at least $(1-\frac{\gamma_1}2)k$. Such a matching exists by Claim~\ref{AC:RBigM} and by Lemma~\ref{lem:matchingBlowUp}.

Observe that $R'_1$ is itself a reduced graph of the partition $V_0'\dcup V_1^1\dcup V_1^2\dcup \ldots\dcup V_k^1\dcup V_k^2$ with respect to the parameters $3\eps$ and
$d_1/2$ and some subgraph $G_1'$ of $G$. In particular, we can
apply Lemma~\ref{lem:Super-regular} to $R_1'$ and the matching $M$ to remove
exactly $3\eps m$ vertices from each cluster of $R_1'$ so that every pair of
clusters corresponding to an edge of $M$ is $(6\eps,d_1/4)$-super-regular while
every pair of clusters corresponding to an edge of $R_1'$  is
$(6\eps,d_1/4)$-regular. It also follows that
every pair of these modified clusters corresponding to an edge of $R'_2$  is $(6\eps,d_2/4)$-regular.

\def\k{{k'}}\def\m{{m'}}
We now move all clusters of $R'_1$ which are not incident to the matching $M$
into the exceptional set. This modification is also performed in the graph
$R'_2$.
Let $\k$ be the number of clusters of this modified graph $R'_1$, and $\m$ be
the size of each of its clusters, which are denoted by $V'_1,\ldots,V'_\k$ (and we write $V'_0$ for the exceptional set).

The modified graph $R'_2$ is obtained from $2\times R_2$ by removing a small number of
vertices. From Claim~\ref{AC:Riron} and Lemma~\ref{lem:conK2} we get that
$2\times R_2$ is $(\beta k/2)$-iron. Therefore
\begin{equation}\label{eq:Riron}
\mbox{$R'_2$ is $(\frac{\beta \k}{5})$-iron.}
\end{equation}

From now on, all references to $R_1',R_2'$ and the matching $M$ will be to these modified versions. 

By the above, there is a partition of the vertices of $G$ into $\k+1$ classes
$V'_0\dcup V'_1\dcup \ldots\dcup V'_\k$, and a spanning subgraph $G'$ of $G$ with the following properties:
\begin{itemize}
\item[(a)] $1/\eps \leq \k \leq 2N(\eps,N',1)\le 2\gamma_1 n$ (using the
bound~\eqref{eq:n0Set}).
\item[(b)] $|V'_0| \leq 2\gamma_1n, |V'_1| =
\cdots = |V'_\k| = \m$.
\item[(c)] $\deg_{G'}(v) \geq \deg_G(v) - 3\gamma_1 n$ for every $v \in
V(G)\setminus V'_0$.
\item[(d)] $G'[V'_i]$ is empty for every $0 \leq i \leq \k$.
\item[(e)] All pairs $(V'_i,V'_j)$ with $1 \leq i < j \leq \k$ are
$6\eps$-regular with density either 0 or at least $d_1/4$.
\item[(f)] There is a $\beta k'/5$-iron graph $R'_2$ on vertex set
$V'_1,\ldots,V'_\k$ such that every edge of $R'_2$
corresponds to a $(6\eps,d_2/4)$-regular pair in $G$.
\item[(g)] There is a perfect matching $M$ on the complete graph formed by
the clusters $V'_1,\ldots,V'_\k$. Further,  every edge of $M$
corresponds to a $(6\eps,d_1/4)$-super-regular pair in $G$.
\end{itemize}

If $XY\in M$, then we call $Y$ the \emph{partner} of $X$. Let us denote the edges of $M$ by $A_iB_i$ for $1 \leqslant i \leqslant \k/2$.
Using Lemma~\ref{lem:ideal} with $\theta = d_1^2$ for each $1 \leqslant i
\leqslant \k/2$, we find $A_i^*$ and $B_i^*$ with $|A_i^*| = |B_i^*| = d_1^2\m$,
such that $(A_i^*,B_i^*)$ is an $(6\eps/d_1^2, d_1^3/16)$-ideal for $(A_i,B_i)$.

We now define the exceptional set $U$ in the statement of the theorem as follows:
\[ U = V'_0\cup \bigcup_{i=1}^{\k/2}(A_i^* \cup B_i^*).\]
Observe that 
\[|U| \leqslant 2\gamma_1 n + d_1^2n \leqslant 3\gamma_1n < \gamma n.\] 
Suppose now that we are in the setting of the theorem, that is, we are given distinct vertices $x_1,y_1,\ldots,x_\ell,y_\ell\in V(G)\setminus U$ (where $1\le \ell\le C$), and our task is to find a system $\mathcal S$ of $\ell$ vertex-disjoint of paths that partition $V(G)$. Furthermore it is required that $x_j$ and $y_j$ are the endvertices of the $j$-th path.

Our first aim is to find a system $\mathcal S' = \{P_1,\ldots,P_{\ell}\}$ of $\ell$ vertex-disjoint paths, with the path $P_j$ having endvertices $x_j$ and $y_j$. We want $\mathcal S'$ to meet the following properties:

\begin{itemize}
\item[(A1)] $V(\mathcal S')$ contains all vertices of $V'_0$;
\item[(A2)] for each $i\in [\k/2]$, we have that $|V(\mathcal S') \cap A_i| =
|V(\mathcal S') \cap B_i|$;
\item[(A3)] for each $i\in [\k/2]$, there is an edge of $\mathcal S'$ whose
respective endvertices lie in $A_i$ and $B_i$;
\item[(A4)] for each $i\in[\k/2]$, we have that $|V(\mathcal S') \cap A_i^*| =
|V(\mathcal S') \cap B_i^*| = 0$.
\end{itemize}
Having obtained this system $\mathcal S'$ we can then find a complete extension
$\mathcal S$ of $\mathcal{S}'$ as follows: For each $i\in[k'/2]$ let $e_i =
a_ib_i$ be an edge of $\mathcal S'$ with $a_i \in A_i$ and $b_i \in B_i$ as
guaranteed by~(A3). Let $A_i'=(A_i \setminus V(\mathcal
S'))\cup\{a_i\}$ and $B_i'=(B_i \setminus V(\mathcal
S'))\cup\{b_i\}$. Since $(A_i^*,B_i^*)$ is an $(\frac{6\eps}{d_1^2},
\frac{d_1^3}{16})$-ideal for $(A_i,B_i)$ and since by property~(A4) the system
$\mathcal S'$ does not meet $A_i^* \cup B_i^*$, we have that the pair  $(A_i',B_i')$ is $(\frac{6\eps}{d_1^2},\frac{d_1^3}{16})$-super-regular. By property (A2) we also have that $|A_i'| = |B_i'|$ so we can apply Lemma~\ref{lem:blow-up} to deduce that $G[A_i'\cup B_i']$ contains a Hamilton path $\tilde{P_i}$ with endvertices $a_i$ and $b_i$. We now replace the edges $e_i$ by the paths $\tilde{P_i}$ for each $1 \leqslant i \leqslant \k/2$ to obtain a new system $\mathcal S$ containing all vertices of $V'_1 \cup \cdots \cup V'_\k$. Since by property (A1) it also contains all vertices of $V'_0$, then $\mathcal S$ is a complete extension of $\mathcal S'$ as asserted by the theorem.

It therefore remains to prove that we can find a system  $\mathcal S'$ satisfying the properties (A1)--(A4). In order to prove that, it will  be actually more convenient to demand $\mathcal S'$ to satisfy the following strengthening  of property~(A2) as well:
\begin{itemize}
\item[(A$2'$)] for each $i\in[\k/2]$, we have that $|V(\mathcal S') \cap A_i| =
|V(\mathcal S') \cap B_i| \leqslant 2C\sqrt{\gamma_1}m'$.
\end{itemize}
Finally we set aside disjoint subsets $D_{x_1},D_{y_1},\ldots,D_{x_{\ell}},D_{y_{\ell}}$ of sizes exactly $d_1^2m'$ each as follows: For each vertex $v$ amongst $x_1,y_1,\ldots,x_{\ell},y_{\ell}$ if $v$ belongs to the cluster $X \in V(R_2')$, and $Y$ is the partner of $X$, then we take any subset $D_v$ of $\neighbor_G(v) \cap (Y \setminus Y^*)$ of size exactly $d_1^2m'$ which is disjoint from any other $D_u$'s already defined. This is possible as by the super-regularity of the pair $(X,Y)$, vertex $v$ has at least $d_1m'/4$ neighbours in $Y$ with at most $d_1^2m'$ of them lying in $Y^*$ and with at most $2C d_1^2 m'$ of them lying in other $D_u$'s. These sets will enable us to have a choice of at least $d_1^2 m'$ vertices when choosing the neighbour of $v$ in the path containing $v$ as an endvertex in the system $\mathcal{S}'$. This can be immediately guaranteed provided that we further demand the following:
\begin{itemize}
\item[(A5)] for each $v \in \{x_1,y_1,\ldots,x_{\ell},y_{\ell}\}$, only the path with endvertex $v$ from the system $\mathcal{S}'$ is allowed to meet $D_v$ and furthermore it is only allowed to meet it at the neighbour of $v$ in this path. 
\end{itemize}
Let us write $D$ for the union of all $D_v$ with $v \in \{x_1,y_1,\ldots,x_{\ell},y_{\ell}\}$ and observe that 
\[ |D| = 2\ell d_1^2 m' \leqslant 2Cd_1^2 m'.\]

We will satisfy (A1) by insisting that the first path $P_1$ of $\mathcal S'$ contains all the vertices of $V_0'$. Let us write $r$ for the size of $V_0'$ and $z_1,\ldots,z_r$ for its vertices. 

\begin{AuxiliaryBipfar}\label{AC:neighborsGarbage}
There are clusters $U_1,W_1,\ldots,U_r,W_r$ of $R_2'$ such that:
\begin{itemize}
\item[(i)] For each $1 \leqslant i \leqslant r$, the cluster $U_i$ is different from $W_i$ and from the partner of $W_i$. 
\item[(ii)] For each $1 \leqslant i \leqslant r$, vertex $z_i$ has at least $\beta m'/2$ neighbours in $U_i \setminus (D \cup U)$.
\item[(iii)] For each $1 \leqslant i \leqslant r$, vertex $z_i$ has at least $\beta m'/2$ neighbours in $W_i \setminus (D \cup U)$.
\item[(iv)] Each cluster appears at most $\sqrt{\gamma_1}m'$ times in the list $U_1,W_1,\ldots,U_r,W_r$. 
\end{itemize}
\end{AuxiliaryBipfar}

\begin{proof}[Proof of Claim~\ref{AC:neighborsGarbage}]
The clusters $U_1,W_1,\ldots,U_r,W_r$ can be chosen greedily. We proceed
sequentially for $i=1,\ldots,r$. At any point we will have chosen at most $2r \le 4 \gamma_1 n$ clusters. So at most 
\[ \frac{4\gamma_1 n}{\sqrt{\gamma_1}m'} \le 5\sqrt{\gamma_1}k'\] 
clusters are not allowed to be chosen again because of (iv). So each $z_i$ has at most $3\gamma_1 n$ neighbours in $U$, at most $2Cd_1^2m'$ neighbours in $D$, at most $5\sqrt{\gamma_1}n$ neighbours in clusters which are not allowed to be chosen because of (iv), and at most $2m'$ neighbours which are not allowed to be chosen because of (i). So at the point when we want to choose $U_i$ or $W_i$, vertex $z_i$ has at least $\beta n/2$  neighbours which belong to sets of the form $V \setminus (U \cup D)$ for some cluster $V$.  But there are at most $k'$ such clusters so there is a choice of cluster which does not violate (i)--(iv).
\end{proof}

For each $1 \leqslant i \leqslant r$ we will choose neighbours $u_i,w_i$ of $z_i$ such that $u_i \in U_i \setminus (D \cup U)$ and $w_i \in W_i \setminus (D \cup U)$. These will not be chosen from the beginning but rather during the construction of the system $\mathcal{S}'$. Suppose that at some point during the construction, we want to choose $u_i \in U_i \setminus (D \cup U)$. By condition (ii) of Claim~\ref{AC:neighborsGarbage} there are $\beta m'/2$ neighbours of $z_i$ in $U_i \setminus (D \cup U)$ which we are allowed to choose from. By (A$2'$) at most $2C\sqrt{\gamma_1}m'$ of those vertices have already been used and so there are at least another $\beta m'/3$ which we can freely choose from. We write $Z_i$ for the partner of $U_i$. The pair $\big( U_i , Z_i \setminus (D \cup U)\big)$ is a subpair of the regular pair $(U_i,Z_i)$, and hence itself regular. Thus, all but $6\eps m'$ of those vertices of $U_i$ that we can still choose from, have at least $d_1m'/8$ neighbours in $Z_i \setminus (D \cup U)$ that have not yet been used in the construction of $\mathcal{S}'$. We also choose $w_i$ similarly. 

Also, for each  $i\in[\k/2]$, we will choose an edge $u_{r+i}w_{r+i}\in E(G)$ such that $u_{r+i}\in A_i\setminus (D \cup U)$ and $w_{r+i}\in B_i\setminus (D \cup U)$. Again these vertices will be not be chosen now but during the construction of the system $\mathcal{S}'$. When choosing them, we will further demand that $u_{r+1}$ has at least $d_1m'/8$ neighbours in $ B_i\setminus (D \cup U)$ which have not yet been used in the construction of $\mathcal{S}'$ and similarly $w_{r+1}$ has at least $d_1m'/8$ neighbours in $ A_i\setminus (D \cup U)$ which have not yet been used in the construction of $\mathcal{S}'$. Again this is possible since $(A_i,B_i)$ is $(6\eps,d_1/4)$-regular. 

Set $r'=r+k'/2$. The bounds $|V'_0|\le 2\gamma_1 n$, and
$k'\le 2\gamma_1 n$ (which is implied by~(a)) give that 
\begin{equation}\label{eq:r'bound}
r'\le 3\gamma_1n\;.
\end{equation}

The system $\mathcal S'=\{P_1,\ldots, P_\ell\}$ will be such that the path $P_1$ will contain all the $2$-paths $u_iz_iw_i$ (for $i=1,\ldots,r$) and all edges $u_iw_i$ (for $i=r+1,\ldots,r'$). Therefore, the path $P_1$ alone will guarantee (A1), i.e.~it will absorb all the vertices of $V'_0$. Further, the path $P_1$ alone will guarantee~(A3), i.e., for every $i\in[\k/2]$ there is an edge of $P_1$ between $A_i$ and $B_i$. We first describe in detail the construction of the path $P_1$.

In order to construct the path $P_1$, for each $0 \leqslant
i \leqslant r'$, we aim to find distinct vertices $u_1,w_1,\ldots,u_{r'},w_{r'}$ as described above and for each $0 \le i \le r'$ a path $Q_i$ in $G$ with endvertices $w_i$ and $u_{i+1}$; here $w_0 = x_1$ and $u_{r'+1}=y_1$. The path $P_1$ will be the union of these paths together with the 2-paths $u_iz_iw_i$ (for $i\in [r]$) and the edges $u_iw_i$ (for $i=r+1,\ldots,r'$). To guarantee that $\mathcal S'$ satisfies properties (A1)--(A5) and (A$2'$) we will require that the paths $Q_i$ satisfy the following properties:

\begin{itemize}
\item[(B1)] the paths $Q_i$ are disjoint and do not contain any vertex from $V'_0$;
\item[(B2)] for each $0 \leqslant i \leqslant r'$ and each $1 \leqslant j \leqslant \k/2$ we have that $|V(Q_i) \cap A_j| = |V(Q_i) \cap B_j|$;
\item[(B3)] for each $1 \leqslant j \leqslant \k/2$, we have that $|V(\cup_i Q_i) \cap A_j|, |V(\cup_i Q_i) \cap B_j| \leqslant 2\sqrt{\gamma_1} \m$;
\item[(B4)] for each $0 \leqslant i \leqslant r'$ and each $1 \leqslant j
\leqslant \k/2$ we have that $|V(Q_i) \cap A_j^*| = |V(Q_i) \cap B_j^*| = 0$;
\item[(B5)] the paths $Q_i$ do not meet any vertices of $D$ with the only possible exceptions being the neighbour of $x_1$ in $Q_0$ and the neighbour of $y_1$ in $Q_{r'+1}$.
\end{itemize}
To achieve these properties we will further demand that the following property is also satisfied:

\begin{itemize}
\item[(B6)] for each $0 \leqslant i \leqslant r'$, the path $Q_i$ has length at most $\gamma_1^{-1/3}$. 
\end{itemize}
Let us now show how this can be done. Suppose we have already chosen the vertices $u_1,w_1,\ldots,u_{i-1},w_{i-1}$, and the paths $Q_0,Q_1,\ldots,Q_{i-1}$ and we are now at the stage where we require to choose vertices $w_i$ to $u_{i+1}$ and the path $Q_i$.

We use~(B6) and~\eqref{eq:r'bound} to infer that the paths
$Q_0,Q_1,\ldots,Q_{i-1}$ contain at most $i \gamma_1^{-1/3} \le r'  \gamma_1^{-1/3}\le 3\gamma_1^{2/3}n$ vertices. In
particular, we have the following.
\begin{AuxiliaryBipfar}\label{AC:ToAvoid}
There are at most $3\gamma_1^{2/3}n/(\gamma_1^{1/2}m') \leqslant 4\gamma_1^{1/6}k'$ indices $j\in[k'/2]$ for which $|V(Q_0\cup Q_1 \cup \cdots \cup
Q_{i-1}) \cap A_j|\geqslant \sqrt{\gamma_1} m'$ or $|V(Q_0\cup Q_1 \cup \cdots \cup Q_{i-1}) \cap B_j| \geqslant \sqrt{\gamma_1} m'$.
\end{AuxiliaryBipfar}
 When
constructing $Q_i$, we will make sure that no vertex of $Q_i$ is contained in such
clusters except possibly the first four and the last four vertices of $Q_i$. (It
might happen that $w_i$ or $u_{i+1}$ belong to such clusters so in this case
$Q_i$ definitely cannot avoid these clusters completely. By using at most four
vertices, and the high minimum degree of $R'_1$ we will be able to get out of these
forbidden clusters and then we will make sure that we never visit them again.)
 If we can achieve this then we can guarantee that for
each $1 \leqslant j \leqslant k'/2$, we have that \[ |V(\cup_i Q_i) \cap A_j|, |V(\cup_i Q_i) \cap B_j| \leqslant \sqrt{\gamma_1} m' + \gamma_1^{-1/3} +
8(r'+1) \leqslant 2\sqrt{\gamma_1}m', \] as required by property (B3).

For finding the paths $Q_i$ we will need to use an auxilary digraph $R^*$, which
should be viewed as a ``shifted version'' of $R'_2$. The vertex set of
$R^*$ is the same as the vertex set of $R'_2$ while its edge set is defined as \[ E(R^*)
= \left\{\overrightarrow{B_jA_i},\overrightarrow{B_iA_j} : A_iA_j \in
E(R'_2)\right\} \cup \left\{\overrightarrow{A_jB_i},\overrightarrow{A_iB_j} :
B_iB_j \in E(R'_2)\right\} \cup
\left\{\overrightarrow{A_jA_i},\overrightarrow{B_iB_j}:A_iB_j \in E(R'_2),i\neq
j\right\}.\]

\begin{AuxiliaryBipfar}\label{AC:R^* highly connected}
The digraph $R^*$ is $\left(\frac{d_2\beta^2 k'}{1000}\right)$-strongly connected.
\end{AuxiliaryBipfar}

\begin{proof}[Proof of Claim~\ref{AC:R^* highly connected}]
Suppose that $R^*$ is not $\left(\frac{d_2\beta^2 k'}{1000}\right)$-strongly connected.
That means that we can write $V(R^*)=S_0\cup S_1\cup S_2$, where
$|S_0|<d_2\beta^2 \k/1000$, $S_1,S_2\neq\emptyset$, and there are no directed
edges from $S_1$ to $S_2$. We partition further each $S_i$ ($i=1,2$) into three sets:
\begin{align*}
S_i^0&=\{X\in S_i: \mbox{partner of $X$ is in $S_0$}\}, \\
S_i^1&=\{X\in S_i\setminus S_i^0: \mbox{partner of $X$ is in $S_{3-i}$}\}, \\
S_i^2&=\{X\in S_i\setminus S_i^0: \mbox{partner of $X$ is in $S_i$}\}\;.
\end{align*}
(See Figure~\ref{fig:Rprime}.)
For the set $L_1=S_0\cup S_1^0\cup S_2^0$ we have
\begin{equation}\label{eq:S0small}
|L_1|\le\frac{d_2\beta^2 k'}{500}\;.
\end{equation}

\def\wdeg{\mathrm{d\overline{eg}}}
The graph $R'_1$ can be viewed as an edge-weighted graph, where the weight of
each its edge is the density of the corresponding regular pair. Thus the weights used
on the edges of $R'_1$ are in the interval $[d_1/4,1]$. In particular, we
have the notion of weighted degree which is defined for a cluster $X\in V(R'_1)$
as the sum of weights of edges incident with $X$, and is denoted $\wdeg(X)$. Observe that the property that all vertices of $G$ have the same degree gets inherited by the
weighted graph $R'_1$, that is, the valency $\Delta$ of the vertices of $G$ corresponds to weighted degrees of approximately $\Delta\frac{k'}{n}$ of the clusters $V_i'$. Taking into account that the relations $n\approx k'm'$ and $\deg_{G'}(v)\approx \Delta$ are only approximate, we get that each cluster $V_i'$, ($i \in [k']$) satisfies
\begin{equation}\label{eq:R'Reg}
(1 - \gamma_2)\frac{\Delta k'}{n} \leqslant \wdeg(V'_i) \leqslant (1 + \gamma_2)\frac{\Delta k'}{n}\;.
\end{equation}

The set $S_2^1$ is independent in $R'_2$ by the definition of the graph
$R^*$. Indeed, suppose that there is an edge $XY\in E(R'_2)$ inside $S_2^1$.
Then, by the definition of $R^*$, there is a directed edge from the partner of $X$ (which is in $S^1_1$) to $Y$, a contradiction to the assumption that there are no directed edges from $S^1_1$ to $S^1_2$. Further, it can be similarly checked that there are no edges between $S_2^1$ and $S_1^2\cup S_2^2$, or between $S_1^2$ and $S_2^2$. This is depicted on Figure~\ref{fig:Rstar}.
\begin{figure}[t]
     \centering
     \subfigure[Separation of the digraph $R^*$. There are no directed edges crossing from left to right. Vertices of $S_0\cup S_1^0\cup S_2^0$ are omitted from the picture.]{
          \label{fig:Rprime}
          \includegraphics[width=.45\textwidth]{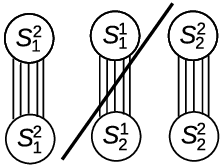}}
     \hspace{.07in}
     \subfigure[The situation in the graph $R'_2$. Allowed edges are
     depicted in grey.]{
          \label{fig:Rstar}
          \includegraphics[width=.45\textwidth]{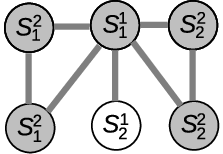}}
     \hspace{.07in}
     \caption{The digraph $R^*$ and the graph $R'_2$. The sets
     $S^2_i$ are split into two according to an arbitrary orientation given by the matching $M$.}
\end{figure}

At this point, we distinguish three cases. Suppose first that $S_1^1=\emptyset$.
Then the set $L_1$ witnesses (using the bound~\eqref{eq:S0small}) that $R'_2$ is not $\left(\frac{d_2\beta^2 k'}{500}\right)$-vertex connected,
and therefore not $\left(\frac{d_2\beta^2
k'}{500}\right)$-iron. This contradicts~\eqref{eq:Riron}. It remains to
consider
\begin{itemize}
\item \underline{Case A:} $S_1^1\neq \emptyset$ and $\max\{|S_1^2|,|S_2^2|\}>
\frac{\beta k'}2$, and
\item \underline{Case B:} $S_1^1\neq \emptyset$ and
$\max\{|S_1^2|,|S_2^2|\}\le\frac{\beta k'}2$.
\end{itemize}
Before diving into Case~A and Case~B separately, we make some calculations
which will turn out to be useful in both cases.

We have
\begin{equation}\label{eq:HS1}
\sum_{W\in S^1_2}\wdeg(W,S^1_1)\ge
 \sum_{W\in S^1_2}\left(\wdeg(W)-|L_1|\right)
 \overset{\eqref{eq:R'Reg},\eqref{eq:S0small}}\ge
 |S^1_2|\left((1- \gamma_2)\frac{\Delta k'}{n}-\frac{d_2\beta^2
 k'}{500}\right)\;.
\end{equation}
Using the facts that $R'_2\subset R'_1$ and that edges of $R'_2$ correspond to 
pairs of density at least $d_2/4$ we have 

\begin{align}\nonumber
e_{R'_2}\left(S_1^2\cup S_2^2,S^1_1\right)+e_{R'_2}\left(S^1_1\right)&\le
\frac4{d_2}\left(\sum_{W\in S^1_1}\wdeg(W)-\sum_{W\in
S^1_1}\wdeg(W,S_2^1)\right)\\ \nonumber &\overset{\eqref{eq:R'Reg}}\le
\frac4{d_2}\left(|S_1^1|(1+ \gamma_2)\frac{\Delta k'}{n}-\sum_{W\in
S^1_1}\wdeg(W,S_2^1)\right)\\
\label{eq:Auxi}
\mbox{[hand-shaking lemma]}\quad&\overset{\eqref{eq:HS1}}\le
\frac{4|S_1^1|}{d_2}\left(2\gamma_2\frac{\Delta k'}n+\frac{d_2\beta^2
k'}{500}\right) \le \frac{\beta^2 k'^2}{100}\;.
\end{align}

We now turn to dealing with Case~A. In this case it is our aim to show that
$R'_2$ is not $(\frac{\beta k'}{10})$-iron.

First, we show that $|S^1_2|>\frac{\beta k'}2$. Indeed, consider $A\in S^1_2$
arbitrary. As $|S_2^1|=|S_1^1|>0$, such an $A$ exists.  As $\deg_{R'_2}(A)\ge(1-2\gamma_2)\frac{\Delta k'}n$, and as $A$ can send edges (in the graph $R'_2$) only to $L_1$ and $S_1^1$, we get 
$$|S_2^1|=|S_1^1|\ge(1-2\gamma_2)\frac{\Delta
k'}n-|L_1|\overset{\eqref{eq:S0small}}> \frac{\beta k'}2\;.$$
We now utilize the assumptions of Case~A. Without loss of generality, assume that
$|S_1^2|>\frac{\beta k'}2$. Set $\rho=\frac{\beta^2}{100}$. The set $L_2=S^1_1$ satisfies by~\eqref{eq:Auxi} that $e_{R'_2}(L_2,V(R'_2)\setminus (L_1\cup L_2))\le \rho
(k')^2$. Further, we have two
disjoint sets $W_1=S^1_2$ and $W_2=S^2_1$ with $\neighbor(W_2) \subseteq L_1 \cup L_2$, and
$\min\{|W_1|,|W_2|\}>2\sqrt{\rho}k'$. Therefore, Lemma~\ref{lem:witNOTiron} applies, and we get that $R'_2$ is not
$(2\sqrt{\rho}k')$-iron. This contradicts~\eqref{eq:Riron}.

\smallskip
It remains to consider Case~B. In this case we get a contradiction by showing
that $G$ is close to a bipartite graph.

Indeed, consider first a partition $W \dcup S_2^1=V(R')$, where
$W=S_1^2\cup S_1^1\cup S_2^2\cup L_1$. The graph $R'_2$ is almost bipartite with respect to
the partition $W\dcup S_2^1$ since $S_2^1$ is independent and
$W$ is very sparse as the following calculation shows:
\begin{align*}
e_{R'_2}(W)&\le k'|L_1| + e_{R'_2}(S_1^2) + e_{R'_2}(S_2^2) + e_{R'_2}(S_1^2,S_2^2) + \left(e_{R'_2}\left(S_1^2\cup S_2^2,S^1_1\right)+e_{R'_2}\left(S^1_1\right)\right)\\
\mbox{[by Case~B,~\eqref{eq:S0small},~\eqref{eq:Auxi}]}\quad&\le
\frac{d_2\beta^2 k'^2}{500} + \frac{\beta^2 k'^2}{8} + \frac{\beta^2 k'^2}{8} + 0 + \frac{\beta^2k'^2}{100}\\
&\le \frac{\beta k'^2}{3}\;.
\end{align*}
The partition $W\dcup S_2^1=V(R'_2)$ induces a partition
$A\dcup B$ of $G$ (placing the  vertices of  $V'_0$ to the sets $A$ and $B$ arbitrarily) such that \[e_G(A)+e_G(B)\le e_{G_2}(A) + e_{G_2}(B) + 2d_2n^2 \le e_{R'_2}(W)(m')^2 + |V'_0|n + 2d_2n^2 <\beta n^2.\]
This is a contradiction to the fact that $G$ is $\beta$-far from bipartiteness.
\end{proof}

Recall that we were looking to choose vertices $w_i$ and $u_{i+1}$ as well as a path $Q_i$ from $w_i$ to $u_{i+1}$. Recall that we want $w_i \in W_i$ and $u_{i+1} \in U_{i+1}$. We write $Z$ for the partner of $U_{i+1}$. By Claim~\ref{AC:ToAvoid} there were at most $4\gamma_1^{1/6}k'$ clusters which we wanted to make sure
that their vertices are avoided by $Q_i$ (except perhaps the first four and last four vertices of $Q_i$). Let us write $S$ for the set of these clusters. 
Since by Claim~\ref{AC:R^* highly connected}  $R^*$ is $\left(\frac{d_2\beta^2 k'}{1000}\right)$-strongly connected and
since also $4\gamma_1^{1/6} k'\ll \frac{d_2\beta^2 k'}{1000}$, we have
that the digraph $R^*-S$ is $\left(\frac{d_2\beta^2
k'}{2000}\right)$-strongly connected. By Lemma~\ref{lem:minDegDiam} there is a directed
path $Q_i'$ in $R^*$ from $W_i$ to $Z$ avoiding $S$ of length 
$t \le \frac{2000}{d_2\beta^2}+1\ll \gamma_1^{-1/3}$.
Suppose $Q_i' = X_1 X_2 \cdots X_t$ where $X_1 = W_i$ and $X_t = Z$. For $i \in
[t]$, let $Y_i$ be the partner of $X_i$. It follows from the definition of
$E(R^*)$ that $Q_i'' = X_1Y_1X_2Y_2\cdots X_tY_t$ is a path in $R'_1$. Observe that by our construction, if a cluster belongs to $S$ then so does its partner. Therefore, since $Q_i'$ avoids $S$, so does $Q_i''$. Observe also that for each $j \in [t]$ the pair $X_jY_j$ is $(6\eps,d_1/4)$-super-regular and for each $j \in [t-1]$ the pair $Y_jX_{j+1}$ is $(6\eps,d_1/4)$-regular. 

We will show how to use $Q_i''$ to find a path 
$Q_i = p_1q_1r_1s_1p_2q_2r_2s_2\cdots p_tq_tr_ts_t$ in $G$, where $p_1 =
w_i,s_t=u_{i+1}$, and for each $j \in [t]$, $p_j,r_j\in X_j$ and $q_i,s_j\in
Y_j$. If we can do this automatically (B2),(B3) and (B6) are satisfied. If we can do it by avoiding all vertices of $Q_1,\ldots,Q_{i-1}$ and all vertices of $U$ then (B1) and (B4) will also be satisfied. By further avoiding all vertices of $D$ except possibly in the case of the neighbours of $x_1$ and $y_1$ then (B5) will also be satisfied.

We will do this as follows: First we will choose $p_1 = w_i \in X_1$ and $s_1 \in Y_1$. The only restrictions apart from the ones mentioned in the previous paragraph are: 
\begin{itemize}
\item[(i)] $p_1 = w_0 = x_1$ if $i=0$, $p_1$ is a neighbour of $z_i$ for $1 \leqslant i \leqslant r$, $p_1$ is a neighbour of $u_i$ (which has been already chosen) if $r+1 \leqslant i \leqslant r'$
\item[(ii)] $p_1$ has at least $d_1m'/8$ neighbours in $Y_1 \setminus (U \cup D)$ which have not yet been used in the construction of $\mathcal{S}'$ except possibly in the case $i=0$.
\item[(iii)] $s_1$ has at least $d_1m'/8$ neighbours in $X_1 \setminus (U \cup D)$ which have not yet been used in the construction of $\mathcal{S}'$
\item[(iv)] $s_1$ has at least $d_1m'/8$ neighbours in $X_2 \setminus (U \cup D)$ which have not yet been used in the construction of $\mathcal{S}'$
\end{itemize}
We have already seen that (i) can be achieved in such a way that (ii) is also satisfied. By the regularity of the pairs $(X_1,Y_1)$ and $(Y_1,X_2)$ conditions (iii) and (iv) give a total of $12\eps m'$ vertices which are not allowed to be chosen. So we can choose such an $s_1$ avoiding also all vertices of $U \cup D$ and all vertices which have been already used. 

Looking now at all neighbours of $p_1$ in $Y_1 \setminus (U \cup D)$ which have not yet been used and all neighbours of $s_1$ in $X_1 \setminus (U \cup D)$ which have not yet been used, by conditions (ii) and (iii) and the regularity of the pair $(X_1,Y_1)$ we can find the vertices $q_1$ and $r_1$ as required except possibly when $i=0$. In this case we can still choose $q_1,r_1$ as required because we are allowed to take any $q_1 \in D_{x_1}$. Since $|D_{x_1}| = d_1^2 m'$ the regularity of the pair $(X_1,Y_1)$ still works to let us find $q_1$ and $r_1$.

An identical argument now works for first finding the vertices $p_2,s_2$ and then the vertices $q_2,r_2$. Here, the equivalent of condition (i) for $p_2$ is that it is a neighbour of $s_1$. Condition (iv) above is the one that lets us choose such a $p_2$ such that the equivalent of condition (ii) for $p_2$ holds. 

The only thing that remains to be checked to complete the argument is what happens with the choice of the last vertex $s_t = u_{i+1}$ of $Q_i$. In the case that $0 \le i \le r-1$ we want $u_{i+1}$ to be a neighbour $z_{i+1}$. We have already seen that the equivalent of condition (iii) for $s_t$ can be guaranteed and there is no need for an equivalent of condition (iv). Finally, if $i=r'$ then $s_t = u_{r'+1} = y_1$ which is already chosen. We can then pick $q_t$ and $r_t$ by allowing $r_t \in D_{y_1}$. The argument is analogous to the one we did in the case $i=0$ for $p_1 = w_0 = x_1$.

So the paths $Q_0,Q_1,\ldots,Q_{r'}$ can be constructed as required and hence so can the path $P_1$. Construction of other paths $P_i$ for $i>1$ again uses the auxiliary graph $R^*$ in the same manner. Recall that for $i > 1$ we want a path $P_i$ from $x_i$ to $y_i$. For its construction, we want to satisfy the following properties:
\begin{itemize}
\item[(C1)] $P_i$ is disjoint from $P_1,\ldots,P_{i-1}$;
\item[(C2)] for each $1 \leqslant j \leqslant \k/2$ we have that $|V(P_i) \cap A_j| = |V(P_i) \cap B_j|$;
\item[(C3)] for each $1 \leqslant j \leqslant \k/2$, we have that $|V(P_i) \cap A_j|, |V(P_i) \cap B_j| \leqslant 2\sqrt{\gamma_1} \m$;
\item[(C4)] for each $1 \leqslant j
\leqslant \k/2$ we have that $|V(P_i) \cap A_j^*| = |V(P_i) \cap B_j^*| = 0$;
\item[(C5)] $P_i$ does not meet any vertex of $D$ with the only possible exception being the neighbours of $x_i$ and $y_i$;
\item[(C6)] $P_i$ has length at most $\gamma_1^{-1/3}$. 
\end{itemize}
These properties are completely analogous to properties (B1)--(B6) we demanded for the paths $Q_j$. (Note that it is not necessary to demand that $P_i$ is disjoint from $V_0'$ as $V_0' \subseteq V(P_1)$.) So the construction of $P_i$ is completely analogous to the construction of the paths $Q_j$ with the only difference being that both the first vertex $x_i$ and the last vertex $y_i$ are fixed. By choosing the neighbour of $x_i$ in $D_{x_i}$ in a similar way as we have chosen the neighbour of $x_1$ in $Q_0$, and by choosing the neighbour of $y_i$ in $D_{y_i}$ in a similar way as we have chosen the neighbour of $y_1$ in $Q_{r'+1}$, we can run the same argument exactly as we did for the paths $Q_j$, unless possibly if $P_i$ must have length 3. In that case we must achieve that the neighbours of $x_i$ and $y_i$ are adjacent in $P_i$. This can be guaranteed as $D_{x_i}$ and $D_{y_i}$ are subsets of a regular pair and their sizes are big enough to guarantee the existence of an edge between them.
\end{proof}


\begin{proof}[Sketch of the proof of Theorem~\ref{thm:RobustOKbip}]
Let $A\dcup B$ be the partition given by Lemma~\ref{lem:almostBipBalanced}. By passing to the subgraph $G[A,B]$ we can assume that the input graph $G$ is bipartite. Remark~\ref{rem:bipartiteISvt} guarantees that this modified graph is still vertex-transitive and Lemma~\ref{lem:almostBipBalanced} guarantees that it has high iron connectivity.

The proof works very similar to the proof of Theorem~\ref{thm:RobustOKbipfar}. We just draw attention to three small differences:

First, the Regularity Lemma must be applied with prepartition $A\dcup B$. Let $\mathcal A$ and $\mathcal B$ be the clusters inside $A$, and $B$, respectively. 


Second, when finding good partners $u_i$ and $w_i$ for exceptional vertex $z_i$, we require that
\begin{equation}\label{eq:reallyAlt}
\mbox{$u_i,v_i\in B$ if $z_i\in A$ and $u_i,v_i\in A$ if $z_i\in B$.} 
\end{equation}

Last, Claim~\ref{AC:R^* highly connected} need not hold in the bipartite
setting. Indeed, typically clusters in $\mathcal A$ form one component and
clusters inside $\mathcal B$ form another component of the auxiliary digraph $R^*$. It can be proven (using the
same methods) that both graphs $R^*[\mathcal A]$ and $R^*[\mathcal B]$ have high
strong connectivity. This is sufficient in the bipartite case. The key for the entire
embedding working is that~\eqref{eq:BIrequirement},~\eqref{eq:reallyAlt} and the
fact that all edges of $M$ cross between $\mathcal{A}$ and $\mathcal{B}$ guarantee that all the paths will automatically occupy the same number of vertices in $A$ as in $B$.
\end{proof}


\section{Proof of Theorem~\ref{thm:main}}\label{sec:proof}

\def\GI{\mathrm{T\ref{lem:goodIslands}}}
\def\NOTBIOK{\mathrm{T\ref{thm:RobustOKbipfar}}}
\def\BIOK{\mathrm{T\ref{thm:RobustOKbip}}}

We first set up constants. Let $\beta_\GI$, $R_\GI$, and $N_0$ be given by Theorem~\ref{lem:goodIslands} for input parameter $\alpha$. Let $N_1$ be given by Theorem~\ref{thm:RobustOKbipfar} for input parameters $\beta_\NOTBIOK=\beta_\GI^4$, $C_\NOTBIOK=R_\GI$, and $\gamma_\NOTBIOK=\frac1{10R_\GI}$. Let $N_2$ be given by Theorem~\ref{thm:RobustOKbip} for input parameters $c_\BIOK=\min\{\beta_\NOTBIOK,\frac{1}{18}\}$ and $C_\BIOK=4R_\GI$. Let
\[
n_0=\max\{N_0,100R_\GI^3,10R_\GI N_1,10R_\GI N_2\}\;.
\]
Suppose now we are in the setting of the theorem.

\smallskip

Consider a partition $V_1\dcup \ldots\dcup V_r$ of $V(G)$ given by
Theorem~\ref{lem:goodIslands}. We have $r<R_\GI$. We call the sets $V_1,\ldots,V_r$
\emph{continents}. If $r=1$ then the existence of a Hamilton cycle follows.
Indeed, consider first the case when $G$ is $c_\BIOK^4$-far from bipartiteness. Let $U_1\subset V(G)$ be the exceptional set given by
Theorem~\ref{thm:RobustOKbipfar}. There exist an edge $xy\in E(G-U_1)$. Using
1-pathitionability of $G$ there exists a Hamilton path from $x$ to $y$. This path
together with the edge $xy$ forms a Hamilton cycle. If on the other hand $G$ is
$c^4_\BIOK$-close to bipartiteness, then an analogous construction using
Theorem~\ref{thm:RobustOKbip} instead of Theorem~\ref{thm:RobustOKbipfar} works.

It remains to consider the case $r>1$. Let $m=|V_1|$. The proof now splits into two cases. The first case deals with the situation
when the graphs $G[V_i]$ are far from bipartiteness. The second case deals with
the setting when the graphs $G[V_i]$ are close to bipartiteness.\footnote{Recall that by Theorem~\ref{lem:goodIslands}, all the graphs $G[V_i]$ have the same distance from being bipartite.}
In both cases one needs to glue paths of the graphs $G[V_i]$ (these paths are guaranteed by pathitionability and bipathitionability, respectively) into one Hamilton cycle.

\smallskip

\noindent\underline{Case~I: \textsl{All the graphs $G[V_i]$ are $c_\BIOK^4$-far from
bipartiteness.}}\\
We write $k=\frac2n\sum_{1\le i<j\le
r}e(V_i,V_j)$. By the symmetry of our partition, each vertex sends exactly $k$
edges outside its own continent. A pair $V_iV_j$ is \emph{fat}
if there exists a matching of size at least $\frac{m}r$ in $G[V_i,V_j]$. If
$e(V_i,V_j)>0$ but $V_iV_j$ is not fat then we say that $V_iV_j$ is \emph{thin}.
Let $k'$ be the number of edges any vertex $v$ sends into thin pairs. By
vertex-transitivity, $k'$ does not depend on the choice of $v$.

\begin{AuxiliaryMain}\label{AC:DenThin}
We have $e(V_i,V_j)<\frac{k'm}r$ for each thin pair $V_iV_j$.
\end{AuxiliaryMain}
\begin{proof}[Proof of Claim~\ref{AC:DenThin}]
Suppose that
\begin{equation}\label{eq:ViVjdense}
e(V_i,V_j)\ge \frac{k'm}r\;.
\end{equation}
We claim that $V_iV_j$ is fat. To this end it suffices by K\"onig's Matching
Theorem to show that there is no vertex cover of $G[V_i,V_j]$ of size less than
$\frac mr$. This is in turn implied by~\eqref{eq:ViVjdense} and by the fact that
$\Delta_G(V_i,V_j)\le k'$.
\end{proof}

\begin{AuxiliaryMain}\label{AC:noThin}
There does not exist any thin pair.
\end{AuxiliaryMain}
\begin{proof}[Proof of Claim~\ref{AC:noThin}]
Let $K$ be the number of edges in thin pairs incident to $V_1$. We have $K=mk'$.
On the other hand, using Claim~\ref{AC:DenThin}, we have $K\le
(r-1)\frac{k'm}r$. Therefore, $mk'\le \frac{r-1}rmk'$, and consequently $k'=0$.
\end{proof}

We construct an auxiliary graph $H$ on the vertex set $\mathcal
V=\{V_1,\ldots,V_r\}$. The edges of $H$ are formed by fat pairs. From the fact
that $G$ is connected, and from Claim~\ref{AC:noThin} we get that $H$ is
connected. Let $T$ be a spanning tree of $H$. Rooting $T$ at its vertex $V_1$ we
get the notion of \emph{children}  of a continent $V_i$, and of
a \emph{parent} $\parent(V_i)$ of $V_i$ (the parent $\parent(V_i)$ is defined only when
$i\neq 1$).

Let $U_1\subset V_1,\ldots,U_r\subset V_r$ be the exceptional sets given by
Theorem~\ref{thm:RobustOKbipfar}. We have $|U_i|<\gamma_\NOTBIOK m$. Each graph $G[V_i]$ is $C_\NOTBIOK$-pathitionable with exceptional set $U_i$. For each fat pair
$V_iV_j$ let $M_{i,j}\subset G[V_i,V_j]$ be a matching of size at least $\frac
mr$.

\begin{AuxiliaryMain}\label{AC:MatchingEdges}
There exists a family $M$ consisting of two matching edges $x^-_{i,j}y^-_{i,j}, x^+_{i,j}y^+_{i,j}$ from each $M_{i,j}$ with $V_iV_j\in E(T)$ and $V_j=\parent(V_i)$ having the following properties:
\begin{itemize}
  \item $x_{i,j}^-,x_{i,j}^+\in V_i$ and $y_{i,j}^-,y_{i,j}^+\in V_j$ for any
  $V_iV_j\in E(T),V_j=\parent(V_i)$,
  \item $M$ is a matching in $G$, and
  \item $V(M)\cap\bigcup_{i=1}^r U_i=\emptyset$.
\end{itemize}
\end{AuxiliaryMain}
\begin{proof}[Proof of Claim~\ref{AC:MatchingEdges}]
The statement follows by greedily choosing two edges  from each matching
$M_{i,j}$ subject to restrictions above. Since the sets $U_i$ and $U_j$ each forbids at most
$\gamma_\NOTBIOK m$ edges of $M_{i,j}$, and the already chosen edges
$x^-_{i',j'}y^-_{i',j'}, x^+_{i',j'}y^+_{i',j'}$ (where $(i',j')\neq(i,j)$)
forbid at most $4(r-1)$ edges, and since we have $2\gamma_\NOTBIOK m+4(r-1)+2\le
|M_{i,j}|$, the choice of $x^-_{i,j}y^-_{i,j}$ and $x^+_{i,j}y^+_{i,j}$ is
possible.
\end{proof}
Given the family $M=\{x^-_{i,j}y^-_{i,j}, x^+_{i,j}y^+_{i,j}\subset
M_{i,j}\}_{V_iV_j\in E(T),V_j=\parent(V_i)}$ from Claim~\ref{AC:MatchingEdges}
we are now ready to construct the desired Hamilton cycle. The first step is to
decompose each continent $V_i$ into a system of paths $\mathcal S_i$. To
describe $\mathcal S_i$ we need to distinguish three cases based on the position
of $V_i$ in $T$.
\begin{itemize}
  \item \textsl{\underline{$V_i$ is the root of $T$ (i.e., $i=1$).}}\\
Let $V_{i_1}, \ldots, V_{i_p}$ be the children of $V_1$. As $p\le r\le C_\NOTBIOK$, we have that $G[V_i]$ is $p$-pathitionable with exceptional set $U_i$. Define $V_{i_{p+1}}=V_{i_1}$. Let $\mathcal S_1$ be a decomposition of $V_1$ into $p$ paths such that the $j$-th path begins in $y^+_{i_{j},1}$ and ends in
$y^-_{i_{j+1},1}$. Such a system of paths exists thanks to the $p$-pathitionability
of $G[V_1]$.
\item \textsl{\underline{$V_i$ is a leaf of $T$, and $i\neq1$.}}\\
Let $V_{i'}$ be the parent of $V_i$. Let $\mathcal S_i$ consist of a
(single) Hamilton path starting in $x_{i,i'}^-$ and ending in $x_{i,i'}^+$. Such
a path exists thanks to the 1-pathitionability
of $G[V_i]$.
\item \textsl{\underline{$V_i$ is an internal vertex of $T$, and $i\neq1$.}}\\
Let $V_{i'}$ be the parent of $V_i$. Let $V_{i_1}, \ldots, V_{i_q}$ be the
children of $V_1$. As $q< r\le C_\NOTBIOK$, we have that $G[V_i]$ is
$(q+1)$-pathitionable with exceptional set $U_i$. Then let $\mathcal S_i$ consist of $q+1$ paths $P_0,P_1,\ldots,P_q$ which decompose $V_i$. We require that $P_0$ has endvertices $x_{i,i'}^+$ and $y_{i_1,i}^+$. The endvertices of the path $P_j$
($j\in[q-1]$) are required to be $y_{i_j,i}^-$ and $y^+_{i_{j+1},i}$. Last, the endvertices of the path $P_q$
are required to be $y_{i_q,i}^-$ and $x^-_{i,i'}$. Such a system of path exists thanks to the $(q+1)$-pathitionability
of $G[V_i]$.   
\end{itemize}
It can be easily checked that $M$ together with the system
$\{\mathcal S_i\}_{i=1}^r$ forms a Hamilton cycle in $G$. See Figure~\ref{fig:concatenating} for an example. 
\begin{figure}[t]
     \centering
     \subfigure[An example of a partition of $G$ into continents $V_1\dcup \ldots\dcup V_4$ together with tree $T$ (depicted in grey), and edges $x^-_{i,j}y^-_{i,j}$, $x^+_{i,j}y^+_{i,j}$.]{
          \label{fig:step1}
          \includegraphics[width=.45\textwidth]{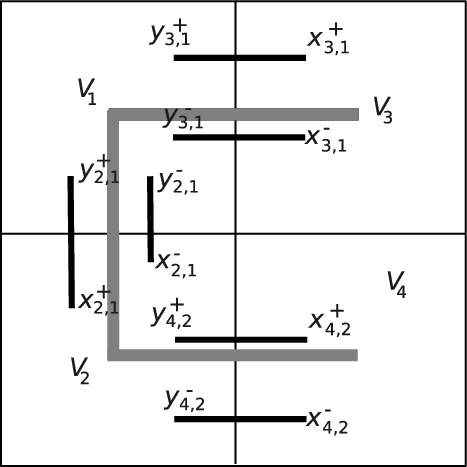}}
     \hspace{.07in}
     \subfigure[The final Hamilton cycle. The systems $\mathcal S_i$ are depicted by dotted lines.]{
          \label{fig:step2}
          \includegraphics[width=.45\textwidth]{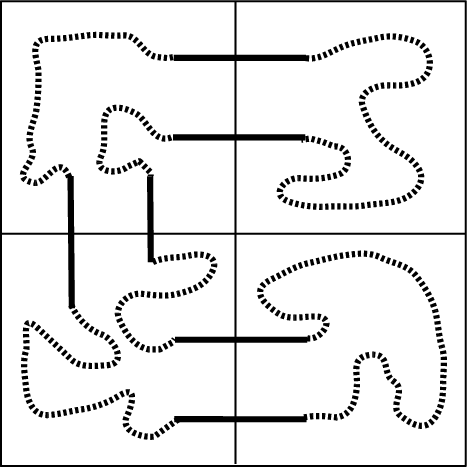}}
     \hspace{.07in}
     \caption{Gluing together the paths $\mathcal S_i$ and $M$.}
     \label{fig:concatenating}
\end{figure}

\bigskip
\noindent\underline{Case~II: \textsl{All the graphs $G[V_i]$ are $c_\BIOK^4$-close to
bipartiteness.}}\\
Let $A_i\dcup B_i$ be the partition of each graph $G[V_i]$ given by Lemma~\ref{lem:almostBipBalanced} with input constant $c_\BIOK$. Let $\mathcal W=\{A_1,B_1,A_2,B_2,\ldots,A_r,B_r\}$. Elements of $\mathcal W$ are called \emph{bicontinents}. A pair $XY$ of elements of $\mathcal W$ is said to be \emph{bifat} if $G[X,Y]$ contains a matching of size at least $\frac{m}{2r}$. If $e(X,Y)>0$ but $XY$ is not bifat then we call the pair $XY$ \emph{bithin}.
\begin{AuxiliaryMain}\label{AC:noBithin}
There does not exist a bithin pair. 
\end{AuxiliaryMain}
\begin{proof}[Proof of Claim~\ref{AC:noBithin}]
The proof translates mutatis mutandis from the proof of Claim~\ref{AC:noThin}.
\end{proof}
Let $H$ be a graph on the vertex set $\mathcal W$, where a pair $XY$ forms an edge of $H$ if $XY$ is bifat. Observe that $A_iB_i\in E(H)$ for every $i\in [r]$. In particular, since $G$ is connected, $H$ is connected as well.

As in Case~I we can find matching $M_{XY}=M_{YX}$ for each $XY\in E(H)$ with the following properties:
\begin{itemize}
\item $M_{XY}\subset G[X,Y]$, $|M_{XY}|=2$,
\item $M=\bigcup_{XY\in E(H)}M_{XY}$ is a matching in $G$, and
\item $V(M)\cap\bigcup_{i=1}^r U_i=\emptyset$.
\end{itemize}
As it will turn out the role of the edges in matchings $M_{A_iB_i}$ is somewhat inferior: they are just used to guarantee connectivity of $H$, and -- unlike other matchings $M_{XY}$ -- they are not guaranteed to lie on the resulting Hamilton cycle. Therefore, we write $M'=M\setminus \bigcup_{i=1}^rM_{A_iB_i}$.

Let $H'$ be a clone of $H$ with each original edge of $H$ replaced by two parallel edges. Since $H'$ is connected and all its degrees are even we can find an Eulerian circuit $\mathcal E$ in $H'$. Also, observe that $H'$ is vertex-transitive, and in particular, we have 
\begin{equation}\label{eq:degAB}
\deg_{H'}(A_i)=\deg_{H'}(B_i)\;,
\end{equation}
for any $i\in [r]$.

The aim is to use $\mathcal E$ to find a Hamilton cycle in $G$. To this end we find requirements for systems of paths $\mathcal S_i$ within each graph $G[V_i]$.

We identify (in a natural way) edges of~$H'$ with edges in~$M$. Therefore, $\mathcal E$ may be viewed as moving between bicontinents. During each (say, $j$-th) visit of $X\in\mathcal W$ we remember vertex $a_{X,j}\in V(M)\cap X$ which was used to enter $X$, and vertex $b_{X,j}\in V(M)\cap X$ which was used to leave $X$. We view $\mathcal E$ cyclically. In other words, for the starting bicontinent $Y$ of the circuit $\mathcal E$ the vertex $b_{Y,1}$ is the vertex coming from the first matching edge along $\mathcal E$ while $a_{Y,1}$ coming from the very last step in $\mathcal E$. 

Let $C_X$ be the number of times bicontinent $X$ was visited. We have $C_X<2r$. Observe also that by~\eqref{eq:degAB} we have $C_{A_i}=C_{B_i}$ for each $i\in[r]$. Therefore, by $4r$-bipathitionability of $G[V_i]$ there exist for each $i\in [r]$ a system of $\mathcal S_i$ of $C_{A_i}+C_{B_i}$ paths decomposing $V_i$ such that:
\begin{itemize}
\item The $j$-th path (for $j\in [C_{A_i}]$) starts in vertex $a_{A_i,j}$ and ends in $b_{A_i,j}$.
\item The $(j+C_{A_i})$-th path (for $j\in [C_{B_i}]$) starts in vertex $a_{B_i,j}$ and ends in $b_{B_i,j}$.
\end{itemize}

It can be easily verified that the system $\{\mathcal S_i\}$ together with the matching $M'$ forms a Hamilton cycle in~$G$. 

\section{Algorithmic aspects}\label{sec:algorithmic}

As said in the Introduction, the problem of deciding whether a graph is
Hamiltonian is NP-hard. Even when the hamiltonicity of a graph $G$
is guaranteed, finding a Hamilton cycle in $G$ cannot be done in polynomial time
unless P=NP. Yet in many situation there is an efficient algorithm for
finding a Hamilton cycle in graphs satisfying certain conditions. See for
example~\cite{BondyChvatal,Sarkozy,CKKO}.

In this short section we note that the tools we use to prove Theorem~\ref{thm:main} can be turned into an efficient algorithm for finding a Hamilton cycle in dense vertex-transitive graphs.

\begin{theorem}\label{thm:algo}
For every $\alpha > 0$ there is an $n_0$ such that every connected vertex-transitive graph on $n \geqslant n_0$ vertices and valency at least $\alpha n$
contains a Hamilton cycle. Moreover there is a polynomial time algorithm for finding a Hamilton cycle in such a graph.
\end{theorem}

Recall the main steps of the proof of Theorem~\ref{thm:main}: 
\begin{enumerate}
\item[(A)] By Theorem~\ref{lem:goodIslands}, the input graph $G$ is partitioned into the continents $V_1\dcup \ldots\dcup V_r$.
\item[(B)] It is checked whether the graphs $G[V_i]$ are close to bipartiteness or not. In the first case, partitions satisfying the conclusion of Lemma~\ref{lem:almostBipBalanced} are found.
\item[(C)] For each $G[V_i]$, an exceptional set $U_i$ is found so that the consequence of Theorem~\ref{thm:RobustOKbipfar} or Theorem~\ref{thm:RobustOKbip} is satisfied. (Depending on whether $G[V_i]$ is far from bipartiteness or not.)
\item[(D)] A way to connect certain systems of paths into one Hamilton cycle in $G$ is devised. (In \textsl{Case~I and Case~II in the proof of Theorem~\ref{thm:main} in the non-bipartite and the bipartite case, respectively.})
\item[(E)] A system of paths (with prescribed end-vertices) is found in the
graphs $G[V_i]$.  (In \textsl{Theorem~\ref{thm:RobustOKbipfar} and Theorem~\ref{thm:RobustOKbip} in the non-bipartite and the bipartite case, respectively}.)
\item[(F)] A Hamilton cycle is found in $G$. (In \textsl{the final part of the
proof of Theorem~\ref{thm:main}}.)
\end{enumerate}

We now discuss the algorithmic versions of the steps above, thus providing a
proof of Theorem~\ref{thm:algo}. 

For step~(A) observe that in the proof of Theorem~\ref{lem:goodIslands} it was
crucial to be able to tell whether a graph is robustly connected. However, the
obvious algorithm for testing robust connectivity requires exponentially many
steps. We can overcome this obstacle with the help of codeg-graphs. We claim
that there is a partition $V_1\dcup \ldots\dcup V_r$ satisfying the conclusion of
Theorem~\ref{lem:goodIslands} and moreover each $V_i$ is a union of components
of the $(19\alpha^2 n/20)$-codeg graph $F$ of $G$. To see this consider the construction of the partition $V_1\dcup\ldots\dcup V_r$ as given by Lemma~\ref{lem:goodIslands-robust}. Using the notation of the proof of Lemma~\ref{lem:goodIslands-robust}, at step $i$, if $G_i$ is not $(\alpha_i^4 n_i/40)$-robust, then we partition $G_i$ into its $(\alpha_i^4 n_i/40)$-islands. By Lemma~\ref{lem:robustlyadjacent}\ref{it:RA2}, every vertex has at most $r_i \alpha_i^4 n_i/40 \leqslant \alpha_i^2 n_i/20$ neighbours outside its island. Therefore, every vertex will have at most
\[
\sum_{i=0}^{\infty} \frac{\alpha_i^2 n_i}{20} = \frac{\alpha^2 }{20} \sum_{i=0}^{\infty} \left( \frac{16}{9} \right)^{i} n_i \leqslant \frac{\alpha^2 n}{20} \sum_{i=0}^{\infty} \left( \frac{8}{9} \right)^{i} = \frac{9 \alpha^2 n}{20}
\]
neighbours outside its continent. In particular, any two vertices which are
neighbours in the $(19\alpha^2 n/20)$-codeg
graph $F$ must belong to the same
continent. There is an efficient way to construct $F$  and moreover by
Lemma~\ref{lem:robustlyadjacent}\ref{it:N1} every component of $F$ has minimum
degree at least $\alpha^2 n/20$ and so $F$ has at most $20/\alpha^2$ components.
In particular, we can construct a bounded number of partitions (depending only
on $\alpha$ and not on $n$) of the vertex set of $G$ by grouping the components
of $F$ in all possible ways. At least one of these partitions satisfies the
conclusion of Theorem~\ref{lem:goodIslands}. From now on the algorithm will work on all these possible partitions concurrently. We will show that for the partition that satisfies the conclusion of Theorem~\ref{lem:goodIslands} it will only take polynomially many steps to construct a Hamilton cycle.
Note that it might happen that some of the partitions do not satisfy the
conclusion of Theorem~\ref{lem:goodIslands}; the algorithm is not required to
produce a Hamilton cycle for these partitions as we only have to produce one
Hamilton cycle.

For step~(B), given a $cn$-iron vertex-transitive graph $G$ we would
like to decide in polynomial time whether it is $c^4$-close to bipartiteness or
not and in the first case exhibit a partition satisfying the conclusion of
Lemma~\ref{lem:almostBipBalanced}. Unfortunately we cannot do this in polynomial
time but not all is lost. Instead, we will show that there is a $0 < c' < c^4$
and a polynomial time algorithm that either proves that $G[V_i]$ is $c'$-far
from bipartiteness or proves that $G[V_i]$ is $c$-close to bipartiteness and
exhibits a partition which satisfies the conclusion of Lemma~\ref{lem:almostBipBalanced}. If it so happens that $G$ is both $c'$-far from and $c^4$-close to bipartiteness then there is no control as to which of the two possible outcomes will appear. To see how this can be done we apply the Regularity Lemma to $G[V_i]$ for some appropriate parameters. It is well known that the partition guaranteed by the Regularity Lemma can be found in polynomial time~\cite{AlDuLeRoYu:AlgorithmicRegularity}. If the reduced graph is not bipartite (this can be checked in constant time) then the counting lemma shows that $G[V_i]$ is far from bipartite. If on the other hand the reduced graph is bipartite then it is immediate that $G[V_i]$ must be close to bipartite. It remains to show how to exhibit a bipartition satisfying the conclusions of Lemma~\ref{lem:almostBipBalanced}. From the reduced graph we can exhibit a partition $A'\dcup B'$ of $G[V_i]$ that satisfies~\eqref{eq:TR}. If every vertex has at least as many neighbours in the opposite part rather than its own part then by Remark~\ref{rem:17} the partition has the required properties. If this was not the case then we move one such vertex to the opposite part and repeat the process. This process has to end (in polynomially many steps) as after each move the number of edges between the two parts strictly increases.

For step~(C) we have already noted that there is an algorithmic version of the
Regularity Lemma~\cite{AlDuLeRoYu:AlgorithmicRegularity}. There are however two
issues that need to be addressed. The first one is that for our proof of
Theorem~\ref{thm:RobustOKbip} it was important that the partition given by the
Regularity Lemma was a refinement of the partition $A\dcup B$ of the vertex set. The
statement of the algorithmic version of the Regularity Lemma
in~\cite{AlDuLeRoYu:AlgorithmicRegularity} does not deal with this issue. From
the proof of the statement however it is immediate that we can start with any
such prepartition. The second issue is that the algorithmic version of the
Regularity Lemma in~\cite{AlDuLeRoYu:AlgorithmicRegularity} is not stated in
the degree form. The usual argument used to deduce the degree form from the standard form is algorithmic provided one knows which pairs are $\eps$-regular.
In principle, it is not easy to check algorithmically whether a pair is
$\eps$-regular or not and in fact the algorithmic proof of the Regularity Lemma
does not say which pair are $\eps$-regular and which are not. It does however
produce a big enough (but possibly) incomplete list of $\eps$-regular pairs and
this is enough for our purpose of constructing a graph of regular pairs $G'$.
The graphs $R_1,R_2,R_1',R_2'$ in the proof of Theorem~\ref{thm:RobustOKbipfar} can now be
easily constructed algorithmically. It is also well-known that there is a
polynomial-time algorithm for finding a maximum matching and so the matching $M$
of $R_1'$ can be constructed. The next step in our proof of
Theorem~\ref{thm:RobustOKbipfar} is an application of
Lemma~\ref{lem:Super-regular} in order to make the pairs corresponding to the
matching $M$ super-regular. We only stated Lemma~\ref{lem:Super-regular} as an
existence result but in the proof of the result one removes from each cluster the
$\eps m$ vertices which have the smallest degree inside its neighbouring cluster
in $M$. Thus this can also be done algorithmically. Finally, we have already
given an algorithmic proof of Lemma~\ref{lem:ideal} and so the exceptional sets
$U_i$ can be constructed in polynomial time.

For step~(D) we observe that the fat or bifat pairs can be easily recognized and so the auxiliary graph $H$ can be constructed efficiently. The global connections in this step are based either on a spanning tree (in the non-bipartite case), or on an Eulerian circuit (in the bipartite case) in $H$. Since $H$ is bounded these can be found in a bounded number of steps. The large matchings between the fat or bifat pairs can also be found in polynomial time and the matching $M$ of Claim~\ref{AC:MatchingEdges} (or the corresponding matching in the bipartite case) is constructed from these matchings greedily.

For step~(E), the system of paths is constructed from the paths
$P_1,\ldots,P_{\ell}$ using the Blow-up Lemma. An algorithmic version of the
Blow-up Lemma appears in~\cite{KoSaSz:BlowUpAlgo}. For the construction of $P_1$
first note that the clusters $U_1,W_1,\ldots,U_r,W_r$ were chosen greedily according to some restrictions. At each step it is easy to verify which clusters are not allowed to be chosen. To complete the construction of $P_1$ we need to construct some auxiliary paths $Q_i$. Each such path was arising from a path $Q_i'$ which was the shortest path in a subdigraph of $R^*$. The digraph $R^*$ and also the set of vertices of $R^*$ which $Q_i'$ is not allowed to pass can be constructed efficiently and hence so can the path $Q_i'$. It is now immediate how to construct the path $Q_i''$ in $R$. In the construction $Q_i = p_1q_1r_1s_1 \cdots p_tq_tr_ts_t$ whenever we were choosing $p_i$ either the choice was already predetermined or we could efficiently obtain a list of allowed vertices to choose as $p_i$. So we can choose $p_i$ greedily from this list and we can also do the same with the choice of $s_i$. Finally, for the choices of $q_i$ and $r_i$ we can again efficiently construct lists of available vertices for $q_i$ and for $r_i$. We want the choice to be such that $q_ir_i$ is an edge and this can be done greedily. The other paths $P_2,\ldots,P_{\ell}$ are constructed in a similar way.

Finally, step~(F) is just putting steps~(D) and~(E) together.


\section{Acknowledgments}
The idea of using the LP-duality in conjunction with the Regularity
Lemma originated in discussion of JH with Dan Kr\'al' and Diana Piguet on
another (yet unpublished) project.

We thank Peter Allen, Michael Krivelevich, L\'aszl\'o Lov\'asz, Igor Pak,  L\'aszl\'o
Pyber, and Bal\'azs Szegedy for useful discussions, and Deryk Osthus for
carefully reading an earlier version of this manuscript.

Finally we thank two of the anonymous referees for their very detailed comments.  
\bibliographystyle{plain} 
\bibliography{bibl}
\end{document}